\documentclass{amsart}

\usepackage{amssymb,latexsym,amsfonts,amsmath}
\usepackage{graphicx}
\usepackage{diagrams}
\usepackage{dsfont}

\newtheorem{theorem}{Theorem}[section]
\newtheorem{lemma}[theorem]{Lemma}

\newtheorem{proposition}[theorem]{Proposition}
\newtheorem{corollary}[theorem]{Corollary}

\newtheorem{definition}[theorem]{Definition}

\newtheorem{remark}[theorem]{Remark}
\numberwithin{equation}{section}

\def\set#1{{\{#1 \}}}
\def\boxspan{\mathit{span}}
\def\params{{\mathsf{q}}}

\newcommand{\R}{{\mathbb{R}}}

\newcommand{\Ze}{{\mathbb Z}}

\newcommand{\N}{{\mathbb{N}}}

\DeclareMathOperator{\Tr}{Tr}
\DeclareMathOperator{\diff}{d}

\newcommand{\ra}{\rightarrow}
\newcommand{\sigalg}{\mathcal{F}}
\newcommand{\filtration}{\mathds{F}}

\newcommand{\ul}{\underline}
\newcommand{\ol}{\overline}
\newcommand{\Let}{:=}
\newcommand{\EE}{\mathds{E}}
\newcommand{\PP}{\mathds{P}}
\newcommand{\traj}[3]{#1_{#2#3}}

\begin{document}

\begin{abstract}
Symbolic approaches to the control design over complex systems employ the construction of finite-state models that are related to the original control systems, 
then use techniques from finite-state synthesis to compute controllers satisfying specifications given in a temporal logic, 
and finally translate the synthesized schemes back as controllers for the concrete complex systems.   
Such approaches have been successfully developed and implemented for the synthesis of controllers over non-probabilistic control systems. 
In this paper, we extend the technique to probabilistic control systems modeled by controlled stochastic differential equations. 
We show that for every stochastic control system satisfying a probabilistic variant of incremental input-to-state stability, 
and for every given precision $\varepsilon > 0$, 
a finite-state transition system can be constructed, 
which is $\varepsilon$-approximately bisimilar (in the sense of moments) to the original stochastic control system. 
Moreover, we provide results relating stochastic control systems to their corresponding finite-state transition systems in terms of probabilistic bisimulation relations known in the literature.   
We demonstrate the effectiveness of the construction by synthesizing controllers for stochastic control systems over rich specifications expressed in linear temporal logic.  
The discussed technique enables a new, automated, correct-by-construction controller synthesis approach for stochastic control systems, 
which are common mathematical models employed in many safety critical systems subject to structured uncertainty and are thus relevant for 
cyber-physical applications.  
\end{abstract}

\title[Symbolic control of stochastic systems via approximately bisimilar finite abstractions]{Symbolic control of stochastic systems via approximately bisimilar finite abstractions}

\author[M. Zamani]{Majid Zamani$^1$} 
\author[P. Mohajerin Esfahani]{Peyman Mohajerin Esfahani$^2$} 
\author[R. Majumdar]{Rupak Majumdar$^3$}
\author[A. Abate]{Alessandro Abate$^1$}
\author[J. Lygeros]{John Lygeros$^2$}
\address{$^1$Delft Center for Systems and Control, Delft University of Technology, 2628 CD, Delft, The Netherlands.}
\email{\{m.zamani,a.abate\}@tudelft.nl}
\urladdr{http://www.dcsc.tudelft.nl/$\sim$\{mzamani,aabate\}}
\address{$^2$Automatic Control Laboratory, ETH Z\"{u}rich, ETL I22, 8092 Z\"{u}rich, Switzerland.}
\email{\{mohajerin,lygeros\}@control.ee.ethz.ch}
\urladdr{http://control.ee.ethz.ch/$\sim$\{peymanm,jlygeros\}}
\address{$^3$Max Planck Institute for Software Systems, 67663 Kaiserslautern, Germany.}
\email{rupak@mpi-sws.org}
\urladdr{http://www.mpi-sws.org/$\sim$rupak}

\maketitle

\section{Introduction, Literature Background, and Contributions}
The design of controllers for complex 
control systems with respect to general temporal specifications in a reliable, 
yet cost-effective way, is a grand challenge in cyber-physical systems research.  
One promising direction is the use of {\em symbolic models}:  
symbolic models are discrete and finite approximations of the continuous dynamics constructed in a way that controllers 
designed for the approximations can be refined to controllers for the original dynamics. 
The relationship between the continuity of the concrete models and the finiteness of their abstractions, 
as well as the interplay between continuous and discrete components that is necessary to show quantitative relations between the two models, 
clearly categorize symbolic approaches within the cyber-physical domain. 

When finite symbolic models exist and can be constructed,   
one can leverage the apparatus of finite-state reactive synthesis \cite{emerson,MalerPnueliSifakis95,Thomas95} 
towards the problem of designing hybrid controllers. 
The formal notion of approximation is captured using $\varepsilon$-approximate bisimulation relations
\cite{girard}, which guarantee that each trace of the continuous system can be
matched by a trace of the symbolic model up to a precision $\varepsilon$, and vice versa.

The effective construction of finite symbolic models has been studied extensively for non-probabilistic control systems. Examples include work on piecewise-affine and multi-affine systems with constant control distribution \cite{bh06,habets}, abstractions based on convexity of reachable sets for sufficiently small sampling time \cite{gunther1}, the use of incremental input-to-state stability \cite{girard2,majid4,pola,pola2,pola1}, non-uniform abstractions of nonlinear systems over a finite-time horizon \cite{tazaki}, and finally sound abstractions for unstable nonlinear control systems \cite{majid}. 
Together with automata-theoretic controller synthesis algorithms \cite{MalerPnueliSifakis95,Thomas95}, 
effective symbolic models form the basis of controller synthesis tools such as \textsf{Pessoa} \cite{pessoa} and \textsf{TuLiP} \cite{tulip}.

However, much less is known about continuous {\em stochastic} control systems.
%
Existing results for probabilistic systems include the construction of finite abstractions for continuous-time 
stochastic dynamical systems under contractivity assumptions \cite{abate},  
for discrete-time stochastic hybrid dynamical systems endowed with certain continuity and ergodicity properties \cite{abate1}, 
and for discrete-time stochastic dynamical systems complying with a notion of bisimulation function \cite{azuma}.     
While providing finite bisimilar abstractions, all the cited techniques are restricted to {\em autonomous} models
(i.e., with no control inputs).
As such, they are of interest for verification purposes, but fall short towards controller synthesis goals.   
On the other hand, for non-autonomous models there exist techniques to {\em check} if an infinite abstraction is 
formally related to a concrete stochastic control system via a notion of stochastic (bi)simulation function \cite{julius1}, 
however these results do not extend to the {\em construction} of approximations nor they deal with {\em finite} abstractions, 
and appear to be computationally tractable only in the autonomous case.  
Further, for specific temporal properties such as (probabilistic) invariance, 
there exist techniques to compute finite abstractions of discrete-time stochastic control systems \cite{AAPLS07}, 
however their generalization to general properties in linear temporal logic is not obvious, 
nor their applicability to continuous-time models.  
The work in \cite{sproston} provides algorithms for the veriÞcation and control problems restricted to probabilistic rectangular automata, 
in which random behaviors occur only over the discrete components -- this limits their application to models with continuous probability laws.  
The work in \cite{LAB09} presents a finite Markov decision process approximation of a continuous-time linear stochastic control systems for 
the verification of given temporal properties, 
however the relationship between abstract and concrete model is not quantitative. 
Along the same lines, 
classical discretization results in the literature \cite{KD01} offer approximations of stochastic control systems that 
are related to the concrete models only asymptotically, rather than according to formal bisimulation or simulation notions that are in the end required to ensure the correspondence of controllers for linear temporal logic specifications over model trajectories.     

Summing up, 
to the best of our knowledge there is no comprehensive work on the construction of finite bisimilar abstractions for 
continuous-time continuous-space stochastic control systems.      
This is unfortunate: 
these systems offer a natural modeling framework for cyber-physical systems operating in an uncertain or noisy environment, 
and automated controller synthesis methodologies can enable more reliable system development at lower costs and times. 

In this paper, we show the existence of $\varepsilon$-approximate bisimilar symbolic models (in the sense of moments) for continuous-time stochastic 
control systems satisfying a probabilistic version of the incremental input-to-state stability property \cite{angeli}, for any given parameter $\varepsilon > 0$.  
The symbolic models are finite if the continuous states lie within a bounded set.  
We also provide a simple way to construct the symbolic abstractions by quantizing the state and input sets. 
By guaranteeing the existence of an $\varepsilon$-approximate bisimulation relation among concrete and abstract models, 
we show that there exists a controller enforcing a desired specification on the symbolic model if and only if 
there exists a controller enforcing an $\varepsilon$-related specification on the original stochastic control system. 
The construction nicely generalizes results for non-probabilistic systems \cite{girard2,majid4,pola}, 
and reduces to these results in the special case of dynamics with no noise. 
Furthermore, we provide quantitative results precisely showing how the proposed $\varepsilon$-approximate bisimulation relations among concrete stochastic control systems and abstract symbolic models are related to known probabilistic bisimulation notions recently developed in the literature. 


We finally illustrate our results on two case studies, where controllers are synthesized over (non-linear) stochastic control systems with respect to rich linear temporal logic specifications of practical relevance. 

\section{Stochastic Control Systems}\label{sec2}
\subsection{Notations} \label{II.A}
The identity map on a set $A$ is denoted by $1_{A}$. If $A$ is a subset of $B$ we denote by \mbox{$\imath_{A}:A\hookrightarrow B$} or simply by $\imath $ the natural inclusion map taking any $a \in A$ to \mbox{$\imath (a) = a \in B$}. The symbols $\N$, $\N_0$, $\Ze$, $\R$, $\R^+$ and $\R_0^+$ denote the set of natural, nonnegative integer, integer, real, positive, and nonnegative real numbers, respectively. The symbols $I_n$, $0_n$, and $0_{n\times{m}}$ denote the identity matrix, the zero vector and zero matrix in $\R^{n\times{n}}$, $\R^n$, and $\R^{n\times{m}}$, respectively. Given a vector \mbox{$x\in\mathbb{R}^{n}$}, we denote by $x_{i}$ the $i$--th element of $x$, and by $\Vert x\Vert$ the infinity norm of $x$, namely \mbox{$\Vert x\Vert=\max\{|x_1|,|x_2|,...,|x_n|\}$}, where $|x_i|$ denotes the absolute value of $x_i$. Given a matrix \mbox{$M=\{m_{ij}\}\in\R^{n\times{m}}$}, we denote by $\Vert{M\Vert}$ the infinity norm of $M$, namely, $\Vert{M}\Vert=\max_{1\leq{i}\leq{n}}\sum_{j=1}^m\vert{m_{ij}\vert}$, and by $\Vert M\Vert_F$ the Frobenius norm of $M$, namely, $\|M\|_F =\sqrt{\text{Tr}\left(MM^T\right)}$, where $\text{Tr}({P})=\sum_{i=1}^np_{ii}$ for any $P=\{p_{ij}\}\in\R^{n\times{n}}$. We denote by $\lambda_{\min}(A)$ and $\lambda_{\max}(A)$ the minimum and maximum eigenvalues of matrix $A$, respectively. The diagonal set $\Delta\subset\R^{2n}$ is defined as: $\Delta=\left\{(x,x) \mid x\in \R^n\right\}$.

The closed ball centered at $x\in{\mathbb{R}}^{n}$ with radius $\varepsilon$ is defined by \mbox{$\mathcal{B}_{\varepsilon}(x)=\{y\in{\mathbb{R}}^{n}\,|\,\Vert x-y\Vert\leq\varepsilon\}$}. A set $B\subseteq \R^n$ is called a 
{\em box} if $B = \prod_{i=1}^n [c_i, d_i]$, where $c_i,d_i\in \R$ with $c_i < d_i$ for each $i\in\set{1,\ldots,n}$.
The {\em span} of a box $B$ is defined as $\boxspan(B) = \min\set{ | d_i - c_i| \mid i=1,\ldots,n}$. Define the $\eta$-approximation $[B]_\eta=\{ b\in B\mid b_i = k_i \eta,~k_i\in\mathbb{Z},~i=1,\ldots,n\}$ for a box $B$ and $\eta \leq \boxspan(B)$.
Note that $[B]_{\eta}\neq\varnothing$ for any $\eta\leq\boxspan(B)$. 
Geometrically, for any $\eta\in{\mathbb{R}^+}$ with $\eta\leq\boxspan(B)$ and $\lambda\geq\eta$, the collection of sets 
\mbox{$\{\mathcal{B}_{\lambda}(p)\}_{p\in [B]_{\eta}}$}
is a finite covering of $B$, i.e., \mbox{$B\subseteq\bigcup_{p\in[B]_{\eta}}\mathcal{B}_{\lambda}(p)$}. 
By defining $[\R^n]_{\eta}=\left\{a\in \R^n\mid a_{i}=k_{i}\eta,~k_{i}\in\mathbb{Z},~i=1,\cdots,n\right\}$, the 
set \mbox{$\bigcup_{p\in[\R^n]_{\eta}}\mathcal{B}_{\lambda}(p)$} is a 
countable covering of $\R^n$ for any $\eta\in\R^+$ and $\lambda\geq\eta$. 
We extend the notions of $\boxspan$ and approximation to finite unions of boxes as follows.
Let $A = \bigcup_{j=1}^M A_j$, where each $A_j$ is a box.
Define $\boxspan(A) = \min\set{\boxspan(A_j)\mid j=1,\ldots,M}$,
and for any $\eta \leq \boxspan(A)$, define $[A]_\eta = \bigcup_{j=1}^M [A_j]_\eta$.

Given a set $X$, a function \mbox{$\mathbf{d}:{X}\times{X}\rightarrow\mathbb{R}_{0}^{+}$} is a metric on $X$ if for any $x,y,z\in{X}$, the following three conditions are satisfied: i) $\mathbf{d}(x,y)=0$ if and only if $x=y$; ii) $\mathbf{d}(x,y)=\mathbf{d}(y,x)$; and iii) (triangle inequality) $\mathbf{d}(x,z)\leq\mathbf{d}(x,y)+\mathbf{d}(y,z)$. 
Given a measurable function \mbox{$f:\mathbb{R}_{0}^{+}\rightarrow\mathbb{R}^n$}, the (essential) supremum of $f$ is denoted by $\Vert f\Vert_{\infty}$; we recall that \mbox{$\Vert f\Vert_{\infty}:=\text{(ess)sup}\{\Vert f(t)\Vert,t\geq0\}$}. A function $f$ is essentially bounded if $\Vert{f}\Vert_{\infty}<\infty$. For a given time $\tau\in\mathbb{R}^+$, define $f_{\tau}$ so that $f_{\tau}(t)=f(t)$, for any $t\in[0,\tau)$, and $f_{\tau}(t)=0$ elsewhere; $f$ is said to be locally essentially bounded if for any $\tau\in\mathbb{R}^+$, $f_{\tau}$ is essentially bounded. A continuous function \mbox{$\gamma:\mathbb{R}_{0}^{+}\rightarrow\mathbb{R}_{0}^{+}$}, is said to belong to class $\mathcal{K}$ if it is strictly increasing and \mbox{$\gamma(0)=0$}; $\gamma$ is said to belong to class $\mathcal{K}_{\infty}$ if \mbox{$\gamma\in\mathcal{K}$} and $\gamma(r)\rightarrow\infty$ as $r\rightarrow\infty$. A continuous function \mbox{$\beta:\mathbb{R}_{0}^{+}\times\mathbb{R}_{0}^{+}\rightarrow\mathbb{R}_{0}^{+}$} is said to belong to class $\mathcal{KL}$ if, for each fixed $s$, the map $\beta(r,s)$ belongs to class $\mathcal{K}_{\infty}$ with respect to $r$ and, for each fixed nonzero $r$, the map $\beta(r,s)$ is decreasing with respect to $s$ and $\beta(r,s)\rightarrow 0$ as \mbox{$s\rightarrow\infty$}. We identify a relation \mbox{$R\subseteq A\times B$} with the map \mbox{$R:A \rightarrow 2^{B}$} defined by $b\in R(a)$ iff \mbox{$(a,b)\in R$}. Given a relation \mbox{$R\subseteq A\times B$}, $R^{-1}$ denotes the inverse relation defined by \mbox{$R^{-1}=\{(b,a)\in B\times A:(a,b)\in R\}$}.

\subsection{Stochastic control systems\label{II.B}}

Let $(\Omega, \sigalg, \PP)$ be a probability space endowed with a filtration $\filtration = (\sigalg_s)_{s\geq 0}$ satisfying the usual conditions of completeness and right continuity \cite[p.\ 48]{ref:KarShr-91}. Let $(W_s)_{s \ge 0}$ be a $p$-dimensional $\filtration$-Brownian motion.

\begin{definition}
\label{Def_control_sys}A \emph{stochastic control system} is a tuple $\Sigma=(\mathbb{R}^{n},\mathsf{U},\mathcal{U},f,\sigma)$, where
\begin{itemize}
\item $\mathbb{R}^{n}$ is the state space;
\item $\mathsf{U}\subseteq\R^m$ is an input set; 
\item $\mathcal{U}$ is a subset of the set of all measurable, locally essentially bounded functions of time from intervals of the form $[0,\infty[$ to $\mathsf{U}$; 
\item $f:\R^n\times\mathsf{U}\rightarrow\R^n$ is a continuous function of its arguments satisfying the following Lipschitz assumption: there exist constants $L_x,L_u\in\R^+$ such that: $\Vert f(x,u)-f(x',u')\Vert\leq L_x\Vert x-x'\Vert +  L_u\Vert u-u'\Vert$ for all $x,x'\in\R^n$ and all $u,u'\in\mathsf{U}$;
\item $\sigma:\R^n\rightarrow\R^{n\times{p}}$ is a continuous function satisfying the following Lipschitz assumption: there exists a constant $Z\in\R^+$ such that: $\Vert\sigma(x)-\sigma(x')\Vert\leq Z\Vert{x}-x'\Vert$ for all $x,x'\in\R^n$.
\end{itemize}
\end{definition}


A stochastic process \mbox{$\xi:\Omega \times [0,\infty[ \rightarrow \mathbb{R}^{n}$} is said to be a \textit{solution process} of $\Sigma$ if there exists $\upsilon\in\mathcal{U}$ satisfying:
\begin{equation}
\label{eq0}
	\diff \xi= f(\xi,\upsilon)\diff t+\sigma(\xi)\diff W_t,
\end{equation}
 $\PP$-almost surely ($\PP$-a.s.), where $f$ is known as the drift, $\sigma$ as the diffusion, and again $W_t$ is Brownian motion. We also write $\xi_{a \upsilon}(t)$ to denote the value of the solution process at time $t\in\R_0^+$ under the input curve $\upsilon$ from initial condition $\xi_{a \upsilon}(0) = a$ $\PP$-a.s., in which $a$ is a random variable that is measurable in $\sigalg_0$. Let us remark that $\sigalg_0$ in general is not a trivial sigma-algebra,  
and thus the stochastic control system $\Sigma$ can start from a random initial condition.  
Let us emphasize that the solution process is uniquely determined, 
since the assumptions on $f$ and $\sigma$ ensure its existence and uniqueness \cite[Theorem 5.2.1, p.\ 68]{oksendal}. 

\section{A Notion of Incremental Stability}\label{sec3}
This section introduces a stability notion for stochastic control systems, 
which generalizes the notion of incremental input-to-state stability ($\delta$-ISS) \cite{angeli} for non-probabilistic control systems. 
The main results presented in this work rely on the stability assumption discussed in this section. 

\begin{definition}
\label{dISS}
A stochastic control system $\Sigma=(\mathbb{R}^{n},\mathsf{U},\mathcal{U},f,\sigma)$ is incrementally input-to-state stable in the $q\textsf{th}$ moment ($\delta$-ISS-M$_q$), where $q\geq1$, if there exist a $\mathcal{KL}$ function $\beta$ and a $\mathcal{K}_{\infty}$ function $\gamma$ such that for any $t\in{\mathbb{R}_0^+}$, any $\R^n$-valued random variables $a$ and $a'$ that are measurable in $\sigalg_0$, and any $\upsilon$, ${\upsilon}'\in\mathcal{U}$, the following condition is satisfied:
\begin{equation}
\EE \left[\left\Vert \xi_{a\upsilon}(t)-\xi_{a'{\upsilon}'}(t)\right\Vert^q\right] \leq\beta\left( \EE\left[ \left\Vert a-a' \right\Vert^q\right], t \right) + \gamma \left(\left\Vert{\upsilon} - {\upsilon}'\right\Vert_{\infty}\right). \label{delta_PISS}
\end{equation}
\end{definition}

It can be easily checked that a $\delta$-ISS-M$_q$ stochastic control system $\Sigma$ is $\delta$-ISS in the absence of any noise as in the following: 
\begin{equation}
\left\Vert \xi_{a\upsilon}(t)-\xi_{a'{\upsilon}'}(t)\right\Vert\leq\beta\left(\left\Vert a-a' \right\Vert, t \right) + \gamma \left(\left\Vert{\upsilon} - {\upsilon}'\right\Vert_{\infty}\right), \label{delta_ISS}
\end{equation}
for $a,a'\in\R^n$, some $\beta\in\mathcal{KL}$, and some $\gamma\in\mathcal{K}_\infty$. Moreover, whenever $f(0_n,0_m)=0_n$ and $\sigma(0_n)=0_{n\times{p}}$ (i.e., the drift and diffusion terms vanish at the origin), then $\delta$-ISS-M$_q$ implies input-to-state stability in the $q\textsf{th}$ moment (ISS-M$_q$) \cite{lirong} and global asymptotic stability in the $q\textsf{th}$ moment (GAS-M$_q$) \cite{debasish}.

%
%


Similar to the use of $\delta$-ISS Lyapunov functions in the non-probabilistic case \cite{angeli}, 
we now describe $\delta$-ISS-M$_q$ in terms of the existence of {\em incremental Lyapunov functions}, 
as defined next.  

\begin{definition}
\label{delta_PISS_Lya}
Consider a stochastic control system $\Sigma=(\mathbb{R}^{n},\mathsf{U},\mathcal{U},f,\sigma)$ 
and a continuous function $V:\mathbb{R}^n\times\mathbb{R}^n\rightarrow\mathbb{R}_0^+$ that is smooth on 
$\{\R^n\times\R^n\}\backslash\Delta$. 
The function $V$ is called an incremental input-to-state stable in the $q\mathsf{th}$ moment ($\delta$-ISS-M$_q$) Lyapunov function for $\Sigma$, 
where $q\geq1$, if there exist $\mathcal{K}_{\infty}$ functions 
$\underline{\alpha}$, $\overline{\alpha}$, $\rho$, and a constant $\kappa\in\mathbb{R}^+$, such that
\begin{itemize}
\item[(i)] $\ul{\alpha}$ (resp. $\ol \alpha$) is a convex (resp. concave) function;
\item[(ii)] for any $x,x'\in\mathbb{R}^n$,\\ 
$\underline{\alpha}\left(\Vert x-x'\Vert^q\right)\leq{V}(x,x')\leq\overline{\alpha}\left(\Vert x-x'\Vert^q\right)$;
\item[(iii)] for any $x,x'\in\mathbb{R}^n$, $x\neq x'$, and for any $u,u'\in\mathsf{U}$,
\begin{align*}
	\mathcal{L}^{u,u'} V(x, x') 
	 &\Let \left[\partial_xV~~\partial_{x'}V\right] \begin{bmatrix} f(x,u)\\f(x',u')\end{bmatrix}+\frac{1}{2} \text{Tr} \left(\begin{bmatrix} \sigma(x) \\ \sigma(x') \end{bmatrix}\left[\sigma^T(x)~~\sigma^T(x')\right] \begin{bmatrix}
\partial_{x,x} V & \partial_{x,x'} V \\ \partial_{x',x} V & \partial_{x',x'} V
\end{bmatrix}	\right) \\
	& \leq-\kappa V(x,x')+\rho(\| u-u'\|),
\end{align*} 
\end{itemize}
where $\mathcal{L}^{u,u'}$ is the infinitesimal generator associated to the stochastic control system \eqref{eq0} \cite[Section 7.3]{oksendal}, which in this case depends on two separate controls $u, u'$. The symbols $\partial_x$ and $\partial_{x,x'}$ denote first- and second-order partial derivatives with respect to $x$ and $x'$, respectively. 
\end{definition}

Roughly speaking, condition $(ii)$ implies that the growth rate of functions $\ol{\alpha}$ and $\ul{\alpha}$ are linear, as a concave function is supposed to dominate a convex one. 
However, these conditions do not restrict the behavior of $\ol\alpha$ and $\ul\alpha$ to only linear functions on a compact subset of $\R^n$.
%
%
Note that the condition $(i)$ is not required in the context of non-probabilistic control systems. 
The following theorem clarifies why such a requirement is instead necessary for a stochastic control system, 
and describes $\delta$-ISS-M$_q$ in terms of the existence of $\delta$-ISS-M$_q$ Lyapunov functions.

\begin{theorem}
\label{the_Lya}
A stochastic control system $\Sigma=(\mathbb{R}^{n},\mathsf{U},\mathcal{U},f,\sigma)$ is $\delta$-ISS-M$_q$ if it admits a $\delta$-ISS-M$_q$ Lyapunov function. 
\end{theorem}

\begin{proof}
	The proof is a consequence of the application of Gronwall's inequality and of Ito's lemma \cite[pp. 80 and 123]{oksendal}. Assume that there exists a $\delta$-ISS-M$_q$ Lyapunov function in the sense of Definition \ref{delta_PISS_Lya}. For any $t\in\R_0^+$, any $\upsilon,\upsilon'\in\mathcal{U}$, and any $\R^n$-valued random variables $a$ and $a'$ that are measurable in $\sigalg_0$, we obtain
		\begin{align*}
			\EE \left[ V(\xi_{a\upsilon}(t),\xi_{a'\upsilon'}(t)) \right]&=\EE \left[V(a,a') + \int_0^{t} \mathcal{L}^{\upsilon(s),\upsilon'(s)} V(\xi_{a\upsilon}(s),\xi_{a'\upsilon'}(s))ds\right]\\
			&\le\EE \left[ V(a,a') + \int_0^{t} \left(-\kappa V(\xi_{a\upsilon}(s),\xi_{a'\upsilon'}(s)) + \rho(\|\upsilon(s)-\upsilon'(s) \|) \right)ds\right] \\
			& \le -\kappa \int_0^{t} \EE \left[ V(\xi_{a\upsilon}(s),\xi_{a'\upsilon'}(s))\right ]ds +  \EE[V(a,a')] + \rho(\|{\upsilon} - \upsilon'\|_\infty)t \nonumber,
		\end{align*}
		which, by virtue of Gronwall's inequality, leads to 
\begin{align}
	\label{V-bound}
		\EE\left[ V(\xi_{a\upsilon}(t),\xi_{a'\upsilon'}(t)) \right]&\le \EE[V(a,a')] \mathsf{e}^{-\kappa t} + t\mathsf{e}^{-\kappa{t}} \rho(\|\upsilon-\upsilon'\|_\infty)\leq  \EE[V(a,a')] \mathsf{e}^{-\kappa t} + \frac{1}{\mathsf{e}\kappa}\rho(\|\upsilon-\upsilon'\|_\infty).
\end{align}
	Hence, using property (ii) in Definition \ref{delta_PISS_Lya}, we have
	\begin{align}\nonumber
		\ul{\alpha}\left(\EE \left[ \|\xi_{a\upsilon}(t)-\xi_{a'\upsilon'}(t)\|^q\right]\right)&\le \EE \left[ \ul{\alpha}\left(\| \xi_{a\upsilon}(t)-\xi_{a'\upsilon'}(t)\|^q\right) \right] \le\EE \left[ V(\xi_{a\upsilon}(t),\xi_{a'\upsilon'}(t))\right] \\\notag&\leq \EE[V(a,a')] \mathsf{e}^{-\kappa t} +\frac{1}{\mathsf{e}\kappa}\rho(\|\upsilon-\upsilon'\|_\infty)\leq \EE\left[\ol{\alpha}\left(\| a - a' \|^q\right)\right] \mathsf{e}^{-\kappa t} + \frac{1}{\mathsf{e}\kappa}\rho(\|\upsilon-\upsilon'\|_\infty)\\ \label{inequl0} 
		&\le \ol{\alpha}\left( \EE\left[\| a - a' \|^q\right] \right) \mathsf{e}^{-\kappa t} +\frac{1}{\mathsf{e}\kappa}\rho(\|\upsilon-\upsilon'\|_\infty),
	\end{align}
	where the first and last inequalities follow from property (i) and Jensen's inequality \cite[p. 310]{oksendal}. Since $\ul{\alpha}\in\mathcal{K}_\infty$, the inequality \eqref{inequl0} yields
\begin{align}\nonumber
\EE \left[ \|\xi_{a\upsilon}(t)-\xi_{a'\upsilon'}(t)\|^q\right]&\le\ul{\alpha}^{-1} \left(\ol{\alpha} \left( \EE\left[ \|a - a' \|^q \right] \right) \mathsf{e}^{-\kappa t} +  \frac{1}{\mathsf{e}\kappa}\rho(\|\upsilon-\upsilon'\|_\infty)\right)\\\notag&\leq\ul{\alpha}^{-1}\left(\ol{\alpha}\left( \EE\left[\| a - a' \|^q\right]\right) \mathsf{e}^{-\kappa t}+\ol{\alpha} \left( \EE \left[ \| a - a' \|^q \right] \right) \mathsf{e}^{-\kappa t}\right)\\\notag&\qquad+\ul{\alpha}^{-1}\left( \frac{1}{\mathsf{e}\kappa}\rho(\|\upsilon-\upsilon'\|_\infty)+ \frac{1}{\mathsf{e}\kappa}\rho(\|\upsilon-\upsilon'\|_\infty)\right)\\\notag&\leq\ul{\alpha}^{-1}\left(2\ol{\alpha} \left( \EE \left[ \| a - a' \|^q \right] \right) \mathsf{e}^{-\kappa t}\right)+\ul{\alpha}^{-1}\left(\frac{2}{\mathsf{e}\kappa}\rho(\|\upsilon-\upsilon'\|_\infty)\right).
\end{align}

	Therefore, by introducing functions $\beta$ and $\gamma$ as
	\begin{align*}
		\beta \left( \EE \left[ \| a - a' \|^q \right],t \right) \Let \ul{\alpha}^{-1}\left(2\ol{\alpha}  \left( \EE \left[ \| a - a' \|^q \right] \right)  \mathsf{e}^{-\kappa t} \right),~~~~~~\gamma\left(\|\upsilon-\upsilon'\|_\infty\right) \Let \ul{\alpha}^{-1}\left(\frac{2}{\mathsf{e}\kappa}\rho(\|\upsilon-\upsilon'\|_\infty)\right),
	\end{align*}
	condition (\ref{delta_PISS}) is satisfied. Hence, the system $\Sigma$ is $\delta$-ISS-M$_q$.	
\end{proof}

One can resort to available software tools, such as \textsf{SOSTOOLS} \cite{prajna}, 
to search for appropriate, non-trivial $\delta$-ISS-M$_q$ Lyapunov functions for system $\Sigma$. 

Now we look into special instances where function $V$ can be easily computed based on the model dynamics. 
The first result provides a sufficient condition for a particular function $V$ to be a $\delta$-ISS-M$_q$ Lyapunov function for a stochastic control system $\Sigma$, when $q=1,2$ (that is, in the first or second moment).

\begin{lemma}\label{lem:lyapunov}
	Consider a stochastic control system $\Sigma=(\mathbb{R}^{n},\mathsf{U},\mathcal{U},f, \sigma)$. Let $q\in\{1,2\}$, $P\in\R^{n\times{n}}$ be a symmetric positive definite matrix, and the function $V: \R^n \times \R^n \ra \R_0^{+}$ be defined as follows:
	\begin{align}
	\label{V}
			V(x,x')\Let\left(\widetilde{V}(x,x')\right)^{\frac{q}{2}}=\left(\frac{1}{q}\left(x-x'\right)^TP\left(x-x'\right)\right)^{\frac{q}{2}},	
	\end{align}
	and satisfy
	\begin{align}
	\label{nonlinear ineq cond}
	(x-x')^T P(f(x,u)-f(x',u))+& \frac{1}{2} \left \| \sqrt{P} \left( \sigma(x) - \sigma(x')\right) \right \|_F^2\le-\widetilde\kappa \left(V(x,x')\right)^{\frac{2}{q}},	
	\end{align}
	or, if $f$ is differentiable with respect to $x$, satisfy
	\begin{align}\label{nonlinear ineq cond1}
		(x-x')^T P\partial_x f(z,u)(x-x') +& \frac{1}{2} \left \| \sqrt{P} \left( \sigma(x) - \sigma(x')\right) \right \|_F^2\le-\widetilde\kappa \left(V(x,x')\right)^{\frac{2}{q}},	
	\end{align}
	for all $x,x',z\in\R^n$, for all $u\in\mathsf{U}$, and for some constant $\widetilde\kappa\in\R^+$. Then $V$ is a $\delta$-ISS-M$_q$ Lyapunov function for $\Sigma$.
\end{lemma}

The proof of Lemma \ref{lem:lyapunov} is provided in the Appendix. 

The next result provides a condition that is equivalent to (\ref{nonlinear ineq cond}) or (\ref{nonlinear ineq cond1}) for linear stochastic control systems $\Sigma$ (that is, for systems with linear drift and diffusion terms) 
in the form of a linear matrix inequality (LMI), which can be easily solved numerically. 
\begin{corollary}\label{corollary}
	Consider a stochastic control system $\Sigma=(\mathbb{R}^{n},\mathsf{U},\mathcal{U},f,\sigma)$, where for all $x\in\R^n$, and $u\in\mathsf{U}$, $f(x,u) \Let Ax + Bu$, for some $A\in\R^{n\times{n}}$, $B\in\R^{n\times{m}}$, and $\sigma(x) \Let\left[\sigma_1 x~\sigma_2 x~\cdots~\sigma_p x\right]$, for some $\sigma_i \in \R^{n \times n}$. Then, 
	function $V$ in \eqref{V} is a $\delta$-ISS-M$_q$ Lyapunov function for $\Sigma$, when $q\in\{1,2\}$, if there exists a constant $\widehat\kappa\in\R^+$ satisfying the following LMI:
	\begin{align}
	\label{LMI}
		PA + A^TP + \sum_{i=1}^p \sigma_i^T P \sigma_i \preceq -\widehat\kappa P.
	\end{align}
\end{corollary}

\begin{proof}
		The corollary is a particular case of Lemma \ref{lem:lyapunov}. It suffices to show that for linear dynamics, the LMI in \eqref{LMI} yields the condition in \eqref{nonlinear ineq cond} or \eqref{nonlinear ineq cond1}. First it is straightforward to observe that
		$$\left \| \sqrt{P} \left( \sigma(x) - \sigma(x')\right) \right \|_F^2 = \text{Tr}\left( \left(\sigma(x)-\sigma(x')\right)^T P \left(\sigma(x)-\sigma(x')\right) \right)=\left(x-x'\right)^T \sum_{i=1}^p \sigma_i^T P \sigma_i(x-x'),$$
		and that 
	$$(x-x')^TP\partial_x f(z,u)(x-x')=\frac{1}{2}(x-x')^T\left(PA+A^TP\right)(x-x'),$$
	for any $x,x',z\in\R^n$ and any $u\in\mathsf{U}$. Now suppose there exists $\widehat\kappa\in\R^+$ such that \eqref{LMI} holds. 
	It can be readily verified that the desired requirements in \eqref{nonlinear ineq cond} and \eqref{nonlinear ineq cond1} are verified by choosing $\widetilde\kappa=\frac{q\widehat\kappa}{2}$.
\end{proof}

As a practical consequence of the previous corollary,  
in order to obtain tighter upper bounds in (\ref{delta_PISS}) one can seek appropriate matrices $P$ by solving the LMI in (\ref{LMI}). 

\begin{remark}
Consider a stochastic control system $\Sigma=(\mathbb{R}^{n},\mathsf{U},\mathcal{U},f,\sigma)$. 
Assume that $f$ is differentiable with respect to $x$ and that, for any $x\in\R^n$, $\sigma(x) \Let\left[\sigma_1 x~\sigma_2 x~\cdots~\sigma_p x\right]$, 
for some $\sigma_i \in \R^{n \times n}$. Then, the function $V$ in \eqref{V} is a $\delta$-ISS-M$_q$ Lyapunov function for $\Sigma$, when $q\in\{1,2\}$, if there exists a constant $\widehat\kappa\in\R^+$ satisfying the following matrix inequality:
\begin{align}\label{contraction}
		P \partial_x f(z,u)+\left(\partial_x f(z,u)\right)^TP+\sum_{i=1}^p \sigma_i^T P \sigma_i  \preceq -\widehat\kappa P,	
	\end{align}
for any $z\in\R^n$ and any $u\in\mathsf{U}$. 
One can easily verify that condition (\ref{contraction}) corresponds to the contractivity conditions (with respect to the states) in \cite{lohmiller,majid1},
obtained with contraction metric $P$, and to the Demidovich's condition in \cite{pavlov} for a system $\Sigma$ in the absence of any noise, 
i.e. $\sigma_i=0_{n\times{n}}$ for all $i=1,\cdots,p$. \hfill$\Box$ 
\end{remark}

\subsection{Noisy and noise-free trajectories}
In order to introduce a symbolic model in Section \ref{existence} for the stochastic control system, 
we need the following technical results, 
which provide an upper bound on the distance (in the $q\textsf{th}$ moment) between the solution processes of $\Sigma$ and those of the corresponding non-probabilistic control system obtained by disregarding the diffusion term (that is, $\sigma$). 
\begin{lemma}\label{lemma3}
Consider a stochastic control system $\Sigma=(\mathbb{R}^{n},\mathsf{U},\mathcal{U},f,\sigma)$ such that $f(0_n,0_m) =0_n$, and $\sigma(0_n) = 0_{n\times{p}}$. 
	Suppose there exists a $\delta$-ISS-M$_q$ Lyapunov function $V$ for $\Sigma$ such that its Hessian is a positive semidefinite matrix in $\R^{2n\times2n}$ and $q\geq2$. Then for any $x$ in a compact set $D$ and any $\upsilon\in\mathcal{U}$, we have 
	\begin{align}\label{mismatch1}
	 	\EE \left[\left\Vert\traj{\xi}{x}{\upsilon}(t)-\traj{\ol \xi}{x}{\upsilon}(t)\right\Vert^q\right] \le h(\sigma,t),	 
	\end{align}
	where {$\ol \xi_{x\upsilon}$ is the solution of the ordinary differential equation (ODE) } $\traj{\dot{\ol \xi}}{x}{\upsilon}(t)=f\left(\traj{\ol \xi}{x}{\upsilon}(t),\upsilon(t)\right)$ starting from the initial condition $x$, and the nonnegative valued function $h$ tends to zero as $t \ra 0$, $t\ra+\infty$, or as $Z\ra0$, where $Z$ is the Lipschitz constant, introduced in Definition \ref{Def_control_sys}.	
\end{lemma}

The proof of Lemma \ref{lemma3} is provided in the Appendix. One can compute explicitly the function $h$ using equation (\ref{up_bound1}) in the Appendix.

\begin{remark}
In the previous lemma, if $V$ is a polynomial, then the condition\footnote{The notation $H(V)(x,x')$ denotes the value of Hessian matrix of $V$ at $(x,x')\in\R^{2n}$.} $H(V)(x,x')\succeq0_{2n\times{2n}}$ for all $x,x'\in\R^n$ is equivalent to $V$ being a convex function \cite{chesi}. Furthermore, if we assume that for all $x,x'\in\R^n$, $H(V)(x,x')=M(x,x')^T M(x,x')$, where $M(x,x')\in\R^{s\times{2n}}$ is a polynomial matrix for some $s\in\N$, then $V$ is a sum of squares which can be efficiently searched through convex linear matrix inequalities optimizations \cite{chesi}, and using some available tools, such as \textsf{SOSTOOLS} \cite{prajna}. \hfill$\Box$
\end{remark}

The following lemma provides an explicit result in line with that of Lemma \ref{lemma3} for a model $\Sigma$ admitting a $\delta$-ISS-M$_q$ Lyapunov function $V$ as in \eqref{V}, where $q\in\{1,2\}$.  
\begin{lemma}
\label{lem:moment est}
	Consider a stochastic control system $\Sigma=(\mathbb{R}^{n},\mathsf{U},\mathcal{U},f,\sigma)$ such that $f(0_n,0_m) =0_n$, and $\sigma(0_n) = 0_{n\times{p}}$. 
	Suppose that the function $V$ in \eqref{V} satisfies (\ref{nonlinear ineq cond}) or (\ref{nonlinear ineq cond1}) for $\Sigma$. For any $x$ in a compact set $D$ and any $\upsilon\in\mathcal{U}$, we have 
	 	$\EE \left[\left\Vert\traj{\xi}{x}{\upsilon}(t)-\traj{\ol \xi}{x}{\upsilon}(t)\right\Vert^2\right] \le h(\sigma,t),$	 
	where 
	\begin{align}\nonumber
	h(\sigma,t)=&\frac{2\left\Vert\sqrt{P}\right\Vert^2n^2\min\{n,p\}Z^2\mathsf{e}^{\frac{-2\widetilde{\kappa}t}{q}}}{\lambda_{\min}^2({P})\widetilde{\kappa}}\left(\lambda_{\max}({P})\left(1-\textsf{e}^{\frac{-\widetilde{\kappa} t}{2}}\right)\sup_{x\in D}\left\{\Vert{x}\Vert^2\right\}+\frac{\left\Vert\sqrt{P}\right\Vert^2L_u^2}{\textsf{e}\widetilde{\kappa}}\sup_{u\in\mathsf{U}}\left\{\Vert{u}\Vert^2\right\}t\right),
	\end{align}
	$Z$ is the Lipschitz constant, as introduced in Definition \ref{Def_control_sys}, and $\ol \xi_{x\upsilon}$ is the solution of the ODE $\traj{\dot{\ol \xi}}{x}{\upsilon}(t)=f\left(\traj{\ol \xi}{x}{\upsilon}(t),\upsilon(t)\right)$ starting from the initial condition $x$. It can be readily verified that the nonnegative valued function $h$ tends to zero as $t \ra 0$, $t\ra+\infty$, or as $Z \ra 0$.	
\end{lemma}

The proof of Lemma \ref{lem:moment est} is provided in the Appendix.

For a linear stochastic control system $\Sigma$, 
the following corollary tailors the result in Lemma \ref{lem:moment est} obtaining possibly a less conservative expression for function $h$, based on parameters of drift and diffusion. 

\begin{corollary}\label{corollary1}
Consider a stochastic control system $\Sigma=(\mathbb{R}^{n},\mathsf{U},\mathcal{U},f,\sigma)$, where for all $x\in\R^n$ and $u\in\R^m$, $f(x)\Let Ax+Bu$, for some $A\in\R^{n\times{n}},B\in\R^{n\times{m}}$, and $\sigma(x) \Let \left[\sigma_{1} x~\sigma_{2}x~\cdots~ \sigma_{p} x\right]$, for some $\sigma_{i} \in \R^{n \times n}$. Suppose that the function $V$ in \eqref{V} satisfies (\ref{LMI}) for $\Sigma$.
For any $x$ in a compact set $D$ and any $\upsilon\in\mathcal{U}$, we have 
	 	$\EE \left[\left\Vert\traj{\xi}{x}{\upsilon}(t)-\traj{\ol \xi}{x}{\upsilon}(t)\right\Vert^2\right] \le h(\sigma,t),$	 
	where 
\begin{align}\nonumber
	h(\sigma,t)=\frac{n\lambda_{\max}\left(\sum\limits_{i = 1}^{p} \sigma^T_iP\sigma_i\right)\mathsf{e}^{-\widehat{\kappa} t}}{\lambda_{\min}({P})}\int_0^t\left(\left\Vert\mathsf{e}^{As}\right\Vert\sup_{x\in{D}}\left\{\left\Vert x\right\Vert\right\}+\left(\int_0^{s}\left\Vert\mathsf{e}^{Ar}B\right\Vert dr\right)\sup_{u\in\mathsf{U}}\left\{\Vert{u}\Vert\right\}\right)^2ds,
	\end{align}
where $\ol \xi_{x\upsilon}$ is the solution of the ODE $\traj{\dot{\ol \xi}}{x}{\upsilon}(t)=A\traj{\ol \xi}{x}{\upsilon}(t)+B\upsilon(t)$ starting from the initial condition $x$.
\end{corollary}

The proof of Corollary \ref{corollary1} is provided in the Appendix.

\section{Symbolic Models}\label{symbolic}

\subsection{Systems}
We employ the notion of system \cite{paulo} to describe both stochastic control systems as well as their symbolic models. 
\begin{definition}
\label{system}
A system $S$ is a tuple 
$S=(X,X_0,U,\longrightarrow,Y,H),$  
where
$X$ is a set of states, 
$X_0\subseteq X$ is a set of initial states, 
$U$ is a set of inputs,
$\longrightarrow\subseteq X\times U\times X$ is a transition relation,
$Y$ is a set of outputs, and
$H:X\rightarrow Y$ is an output map.
\end{definition}

A transition \mbox{$(x,u,x')\in\longrightarrow$} is also denoted by $x\rTo^ux'$. For a transition $x\rTo^ux'$, state $x'$ is called a \mbox{$u$-successor}, or simply a successor, of state $x$. 
For technical reasons, we assume that for any $x\in X$ and $u\in U$, there exists some $u$-successor of $x$ ---  
let us remark that this is always the case for the considered systems later in this paper. 

System $S$ is said to be
\begin{itemize}
\item \textit{metric}, if the output set $Y$ is equipped with a metric
$\mathbf{d}:Y\times Y\rightarrow\mathbb{R}_{0}^{+}$;
\item \textit{finite}, if $X$ is a finite set;
\item \textit{deterministic}, if for any state $x\in{X}$ and any input $u$, there exists exactly one \mbox{$u$-successor}. 
\end{itemize}

For a system $S=(X,X_0,U,\longrightarrow,Y,H)$ and given any state $x_0\in{X_0}$, a finite state run generated from $x_0$ is a finite sequence of transitions: $x_0\rTo^{u_0}x_1\rTo^{u_1}x_2\rTo^{u_2}\cdots\rTo^{u_{n-2}}x_{n-1}\rTo^{u_{n-1}}x_n,$
such that $x_i\rTo^{u_i}x_{i+1}$ for all $0\leq i<n$. A finite state run can be directly extended to an infinite state run as well.  

\subsection{System relations}
We recall the notion of approximate (bi)simulation relation, introduced in \cite{girard}, 
which is useful when analyzing or synthesizing controllers for deterministic systems.  
 
\begin{definition}\label{APSR}
Let \mbox{$S_{a}=(X_{a},X_{a0},U_{a},\rTo_{a},Y_a,H_{a})$} and 
\mbox{$S_{b}=(X_{b},X_{b0},U_{b},\rTo_{b},Y_b,H_{b})$} be metric systems with the 
same output sets $Y_a=Y_b$ and metric $\mathbf{d}$.
For $\varepsilon\in\mathbb{R}^{+}$, 
a relation 
\mbox{$R\subseteq X_{a}\times X_{b}$} is said to be an $\varepsilon$-approximate simulation relation from $S_{a}$ to $S_{b}$ 
if the following three conditions are satisfied:
\begin{itemize}
\item[(i)] for every $x_{a0}\in{X_{a0}}$, there exists $x_{b0}\in{X_{b0}}$ with $(x_{a0},x_{b0})\in{R}$;
\item[(ii)] for every $(x_{a},x_{b})\in R$ we have \mbox{$\mathbf{d}(H_{a}(x_{a}),H_{b}(x_{b}))\leq\varepsilon$};
\item[(iii)] for every $(x_{a},x_{b})\in R$ we have that \mbox{$x_{a}\rTo_{a}^{u_a}x'_{a}$ in $S_a$} implies the existence of \mbox{$x_{b}\rTo_{b}^{u_b}x'_{b}$} in $S_b$ satisfying $(x'_{a},x'_{b})\in R$.
\end{itemize}  
A relation $R\subseteq X_a\times X_b$ is said to be an $\varepsilon$-approximate bisimulation relation between $S_a$ and $S_b$
if $R$ is an $\varepsilon$-approximate simulation relation from $S_a$ to $S_b$ and
$R^{-1}$ is an $\varepsilon$-approximate simulation relation from $S_b$ to $S_a$.

System $S_{a}$ is $\varepsilon$-approximately simulated by $S_{b}$, or $S_b$ $\varepsilon$-approximately simulates $S_a$, 
denoted by \mbox{$S_{a}\preceq_{\mathcal{S}}^{\varepsilon}S_{b}$}, if there exists
an $\varepsilon$-approximate simulation relation from $S_{a}$ to $S_{b}$.
System $S_{a}$ is $\varepsilon$-approximate bisimilar to $S_{b}$, denoted by \mbox{$S_{a}\cong_{\mathcal{S}}^{\varepsilon}S_{b}$}, if there exists
an $\varepsilon$-approximate bisimulation relation $R$ between $S_{a}$ and $S_{b}$.
\end{definition}

Note that when $\varepsilon=0$, 
condition (ii) in the above definition is modified as $(x_a,x_b)\in R$ if and only if 
$H_{a}(x_{a})=H_{b}(x_{b})$, and $R$ becomes an exact  simulation relation, as introduced in \cite{milner}. 
Similarly, whenever $\varepsilon=0$, 
$R$ possibly becomes an exact bisimulation relation.

\section{Symbolic Models for Stochastic Control Systems}\label{existence}
This section contains the main contribution of the paper. 
We show that for any stochastic control system $\Sigma$ admitting a $\delta$-ISS-M$_q$ Lyapunov function, 
and for any precision level $\varepsilon\in\R^+$, we can construct a finite system that is $\varepsilon$-approximately bisimilar to $\Sigma$. 
In order to do so, we use the notion of system as an abstract representation of a stochastic control system, capturing all the information contained in it.  
More precisely, given a stochastic control system $\Sigma=(\mathbb{R}^{n},\mathsf{U},\mathcal{U},f,\sigma)$, 
we define an associated metric system $S(\Sigma)=(X,X_{0},U,\rTo,Y,H),$ 
where:
\begin{itemize}
\item $X$ is the set of all $\R^n$-valued random variables defined on the probability space 
$(\Omega,\sigalg,\PP)$;
\item $X_{0}$ is the set of all $\R^n$-valued random variables that are measurable over the trivial sigma-algebra $\sigalg_0$, i.e., 
the system starts from a deterministic initial condition, 
which is equivalently a random variable with a Dirac probability distribution;
\item $U=\mathcal{U}$;
\item $x\rTo^{\upsilon} x'$ if $x$ and $x'$ are measurable in $\sigalg_{t}$ and $\sigalg_{t+\tau}$, respectively, for some $t \in \R^+_0$ and $\tau\in\R^+$, and 
there exists a solution process $\xi:\Omega\times\R_0^+\rightarrow\R^n$ of $\Sigma$ satisfying $\xi(t) = x$ and $\xi_{x \upsilon}(\tau) = x'$ $\PP$-a.s.;
\item $Y$ is the set of all $\R^n$-valued random variables defined on 
the probability space 
$(\Omega,\sigalg,\PP)$;
\item $H=1_{X}$.
\end{itemize}
We assume that the output set $Y$ is equipped with the natural metric $\mathbf{d}(y,y')=\left(\EE\left[\left\Vert y-y'\right\Vert^q\right]\right)^{\frac{1}{q}}$, for any $y,y'\in{Y}$ and some $q\geq1$. Let us remark that the set of states of $S(\Sigma)$ is uncountable and that $S(\Sigma)$ is a deterministic system in the sense of Definition \ref{system}, since (cf. Subsection \ref{II.B})  
the solution process of $\Sigma$
is uniquely determined. 

The results in this section rely on additional assumptions on model $\Sigma$ that are described next  
(however such assumptions are not required for the definitions and results in Sections \ref{sec2}, \ref{sec3}, and \ref{symbolic}).  
We restrict our attention to stochastic control systems \mbox{$\Sigma=(\mathbb{R}^{n},\mathsf{U},\mathcal{U},f,\sigma)$} with $f(0_n,0_m) = 0_n$, $\sigma(0_n) = 0_{n\times{p}}$, and input sets $\mathsf{U}$ that are assumed to be finite unions of boxes. 
We further restrict our attention to sampled-data stochastic control systems, where input curves belong to set $\mathcal{U}_\tau$ which contains only curves that are constant over intervals of length $\tau\in\R^+$, i.e.
$$ \mathcal{U}_\tau=\Big\{\upsilon:\R_0^+\to \mathsf{U}\,\,\vert\,\,\upsilon(t)=\upsilon((k-1)\tau), t\in[(k-1)\tau,k\tau[, k\in\N\Big\}.$$
 
Let us denote by $S_\tau(\Sigma)$ a sub-system of $S(\Sigma)$ obtained by selecting those transitions from $S(\Sigma)$ corresponding to solution processes of duration $\tau$ and to control inputs in $\mathcal{U}_\tau$. 
This can be seen as the time discretization or as the sampling of a process. 
More precisely, 
given a stochastic control system $\Sigma=(\mathbb{R}^{n},\mathsf{U},\mathcal{U}_\tau,f,\sigma)$, 
we define the associated metric system $S_\tau(\Sigma)=\left(X_\tau,X_{\tau0},U_\tau,\rTo_\tau,Y_\tau,H_\tau\right)$, 
where $X_\tau=X$, $X_{\tau0}=X_{0}$, $U_\tau=\mathcal{U}_\tau$, $Y_\tau=Y$, $H_\tau=H$, and 
\begin{itemize}
\item $x_\tau\rTo^{\upsilon_\tau}_\tau{x'_\tau}$ if $x_\tau$ and $x'_\tau$ are measurable, 
respectively, in $\sigalg_{k\tau}$ and $\sigalg_{(k+1)\tau}$ for some $k \in \N_0$, and 
there exists a solution process $\xi:\Omega\times\R_0^+\rightarrow\R^n$ of $\Sigma$ satisfying $\xi(k\tau) = x_\tau$ and $\xi_{x_\tau\upsilon_\tau}(\tau) = x'_\tau$ $\PP$-a.s..
\end{itemize}

Notice that a finite state run 
$x_{0}\rTo^{\upsilon_{0}}_{\tau}x_{1}\rTo^{\upsilon_1}_{\tau} \,...\, \rTo^{\upsilon_{N-1}}_{\tau} x_{N},$ of $S_{\tau}(\Sigma)$, where $\upsilon_i\in\mathcal{U}_\tau$ and $x_i=\xi_{x_{i-1}\upsilon_{i-1}}(\tau)$ for $i=1,\cdots,N$, captures the solution process of the stochastic control system $\Sigma$ at times $t=0,\tau,\cdots,N\tau$,  
started from the deterministic initial condition $x_{0}$ and resulting from a control input $\upsilon$ obtained by the concatenation of the input curves  
$\upsilon_{i}$ \big(i.e. $\upsilon(t)=\upsilon_{i-1}(t)$ for any $t\in [(i-1)\tau,i\,\tau[$\big), for $i=1,\cdots,N$.

Let us proceed introducing a fully symbolic system for the concrete model $\Sigma$. 
Consider a stochastic control system $\Sigma=(\mathbb{R}^{n},\mathsf{U},\mathcal{U}_\tau,f,\sigma)$ and a triple $\mathsf{q}=(\tau,\eta,\mu)$ of quantization parameters, where $\tau$ is the sampling time, $\eta$ is the state space quantization, and $\mu$ is the input set quantization. 
Given $\Sigma$ and $\mathsf{q}$, consider the following system:
\begin{equation}\label{T2}
S_{\mathsf{q}}(\Sigma)=(X_{\mathsf{q}},X_{\params0},U_{\mathsf{q}},\rTo_{\mathsf{q}},Y_{\mathsf{q}},H_{\mathsf{q}}),
\end{equation}
consisting of (cf. Notation in Subsection \ref{II.A}):
\begin{itemize}
\item $X_{\mathsf{q}}=[\R^n]_\eta$;
\item $X_{\params0}=[\R^n]_\eta$;
\item $U_{\mathsf{q}}=[\mathsf{U}]_\mu$;
\item $x_{\mathsf{q}}\rTo_{\mathsf{q}}^{u_{\mathsf{q}}}x'_{\mathsf{q}}$ if there exists a $x'_\params\in X_\params$ such that $\left\Vert\ol{\xi}_{x_{\mathsf{q}}u_{\mathsf{q}}}(\tau)-x'_{\mathsf{q}}\right\Vert\leq\eta$, where $\dot{\ol{\xi}}_{x_\params u_\params}(t)=f\left(\ol{\xi}_{x_\params u_\params}(t),u_\params(t)\right)$;
\item $Y_{\mathsf{q}}$ is the set of all $\R^n$-valued random variables defined on the probability space $(\Omega,\sigalg,\PP)$;
\item $H_{\mathsf{q}}=\imath:X_{\mathsf{q}}\hookrightarrow Y_{\mathsf{q}}$.
\end{itemize}

Note that we have abused notation by identifying $u_{\mathsf{q}}\in[\mathsf{U}]_\mu$ with the constant input curve with domain $[0,~\tau[$ and value $u_{\mathsf{q}}$. Notice that the proposed abstraction $S_\params(\Sigma)$ is a deterministic system in the sense of Definition \ref{system}. In order to establish an approximate bisimulation relation, the output set $Y_\params$ is defined similarly to the stochastic system $S_\tau(\Sigma)$. Therefore, in the definition of $H_\params$, the inclusion map $\imath$ is meant, with a slight abuse of notation, as a mapping from a deterministic grid point to a random variable with a Dirac probability distribution centered at the grid point. As argued in \cite{paulo}, there is no loss of generality to alternatively assume that $Y_\params=X_\params$ and $H_\params=1_{X_\params}$. For later use, we denote by $\mathsf{X}_{x_{\params0}u}:\N_0\ra X_\params$ a state run of $S_\params(\Sigma)$ from initial condition $\mathsf{X}_{x_{\params0}u}(0)=x_{\params0}$ and under input sequence $u:\N_0\ra[\mathsf{U}]_\mu$.

The transition relation of $S_\params(\Sigma)$ is well defined in the sense that for every 
$x_\params\in[\R^n]_\eta$ and every $u_\params\in[\mathsf{U}]_\mu$ there always exists $x'_\params\in[\R^n]_\eta$ 
such that $x_\params\rTo_\params^{u_\params}x'_\params$. 
This can be seen since by definition of $[\R^n]_\eta$, 
for any $\widehat{x}\in \R^n$ there always exists a state $\widehat{x}'\in [\R^n]_\eta$ such 
that $\Vert \widehat{x}-\widehat{x}'\Vert\leq\eta$. 
Hence, for $\ol{\xi}_{x_\params u_\params}(\tau)$ there always exists a state 
$x'_\params\in [\R^n]_\eta$ satisfying $\left\Vert\ol{\xi}_{x_{\mathsf{q}}u_{\mathsf{q}}}(\tau)-x'_{\mathsf{q}}\right\Vert\leq\eta$.

In order to show the first main result of this work, 
we raise a supplementary assumption on the $\delta$-ISS-M$_q$ Lyapunov function $V$ as follows: 
\begin{equation}\label{supplement}
\vert V(x,y)-V(x,z)\vert\leq\widehat\gamma(\Vert y-z\Vert),
\end{equation}
for any $x,y,z\in\R^n$, and some $\mathcal{K}_\infty$ and concave function $\widehat\gamma$. This assumption is not restrictive, provided $V$ is restricted to a compact subset of $\R^n\times\R^n$. Indeed, for all $x,y,z\in D$, where $D\subset\R^n$ is compact, by applying the mean value theorem to the function $y\rightarrow V(x,y)$, one gets $$\left\vert V(x,y)-V(x,z)\right\vert\leq\widehat\gamma\left(\Vert y-z\Vert\right),~\text{where}~~\widehat\gamma({r})=\left(\max_{x,y\in D\backslash\Delta}\left\Vert\frac{\partial{V}(x,y)}{\partial{y}}\right\Vert\right)r.$$ 
In particular, for the $\delta$-ISS-M$_1$ Lyapunov function $V$ defined in (\ref{V}), 
we obtain explicitly that $\widehat\gamma({r})= \frac{\lambda_{\max}\left(P\right)}{\sqrt{\lambda_{\min}\left(P\right)}} r$ \cite[Proposition 10.5]{paulo}. 
We can now present the first main result of the paper, 
which relates the existence of a $\delta$-ISS-M$_q$ Lyapunov function to the construction of a symbolic model.  
\begin{theorem}\label{main_theorem}
Consider a stochastic control system $\Sigma=(\mathbb{R}^{n},\mathsf{U},\mathcal{U}_\tau,f,\sigma)$, admitting a $\delta$-ISS-M$_q$ Lyapunov function $V$, of the form of (\ref{V}) or the one explained in Lemmas \ref{lemma3}. For any $\varepsilon\in\R^+$ and any triple $\mathsf{q}=(\tau,\eta,\mu)$ of quantization parameters satisfying $\mu\leq\boxspan(\mathsf{U})$ and
\begin{align}\label{bisim_cond1}
\overline\alpha\left(\eta^q\right)\leq\underline\alpha\left(\varepsilon^q\right),\\\label{bisim_cond}
\mathsf{e}^{-\kappa\tau}\underline\alpha\left(\varepsilon^q\right)+\frac{1}{\mathsf{e}\kappa}\rho(\mu)+\widehat\gamma\left(\left(h(\sigma,\tau)\right)^{\frac{1}{q}}+\eta\right)\leq\underline\alpha\left(\varepsilon^q\right),
\end{align}
we have that \mbox{$S_{\mathsf{q}}(\Sigma)\cong_{\mathcal{S}}^{\varepsilon}S_{\tau}(\Sigma)$}.
\end{theorem} 


It can be readily seen that when we are interested in the dynamics of $\Sigma$, initialized on a compact $D\subset\R^n$ of the form of finite union of boxes and for a given precision $\varepsilon$, 
there always exists a sufficiently large value of $\tau$ and small enough values of $\eta$ and $\mu$, such that $\eta\leq\boxspan(D)$ and such that the conditions in (\ref{bisim_cond1}) and (\ref{bisim_cond}) are satisfied. 
On the other hand, for a given fixed sampling time $\tau$, one can find sufficiently small values of $\eta$ and $\mu$ satisfying $\eta\leq\boxspan(D)$, (\ref{bisim_cond1}) and (\ref{bisim_cond}), as long as the precision $\varepsilon$ is lower bounded by:
\begin{equation}\label{lower_bound}
\varepsilon>\left(\ul\alpha^{-1}\left(\frac{\widehat\gamma\left(\left(h(\sigma,\tau)\right)^{\frac{1}{q}}\right)}{1-\mathsf{e}^{-\kappa\tau}}\right)\right)^{\frac{1}{q}}.
\end{equation}
One can easily verify that the lower bound on $\varepsilon$ in (\ref{lower_bound}) goes to zero as $\tau$ goes to infinity or as $Z \ra 0$, where $Z$ is the Lipschitz constant, introduced in Definition \ref{Def_control_sys}. Furthermore, one can try to minimize the lower bound on $\varepsilon$ in (\ref{lower_bound}) by appropriately choosing a $\delta$-ISS-M$_q$ Lyapunov function $V$ (cf. Section \ref{sec3}).

\begin{proof}
We start by proving that \mbox{$S_{\tau}(\Sigma)\preceq^{\varepsilon}_\mathcal{S}S_{\params}(\Sigma)$}.
Consider the relation $R\subseteq X_{\tau}\times X_{\params}$ defined by 
{$\left(x_{\tau},x_{\params}\right)\in R$}
if and only if 
{$\mathbb{E}\left[V\left(H_{\tau}(x_{\tau}),H_{\params}(x_{\params})\right)\right]=\mathbb{E}\left[V\left(x_{\tau},x_{\params}\right)\right]\leq\underline\alpha\left(\varepsilon^q\right)$}. 
Since \mbox{$X_{\tau 0}\subseteq\bigcup_{p\in[\mathbb{R}^n]_{\eta}}\mathcal{B}_{\eta}(p)$}, 
for every $x_{\tau 0}\in{X_{\tau 0}}$ there always exists \mbox{$x_{\params 0}\in{X}_{\params 0}$} such that $\Vert{x_{\tau0}}-x_{\params0}\Vert\leq\eta$. Then,
\begin{equation}\nonumber
\mathbb{E}\left[V({x_{\tau0}},x_{\params0})\right]=V({x_{\tau0}},x_{\params0})\leq\overline\alpha(\Vert x_{\tau0}-x_{\params0}\Vert^q)\leq\overline\alpha\left(\eta^q\right)\leq\underline\alpha\left(\varepsilon^q\right),
\end{equation}
because of (\ref{bisim_cond1}) and since $\overline\alpha$ is a $\mathcal{K}_\infty$ function.
Hence, \mbox{$\left(x_{\tau0},x_{\params0}\right)\in{R}$} and condition (i) in Definition \ref{APSR} is satisfied. Now consider any \mbox{$\left(x_{\tau},x_{\params}\right)\in R$}. Condition (ii) in Definition \ref{APSR} is satisfied because
\begin{equation} 
\label{convexity}
\left(\mathbb{E}\left[\Vert x_{\tau}-x_{\params}\Vert^q\right]\right)^{\frac{1}{q}}\leq\left(\underline\alpha^{-1}\left(\mathbb{E}\left[V(x_{\tau},x_{\params})\right]\right)\right)^{\frac{1}{q}}\leq\varepsilon.
\end{equation}
We used the convexity assumption of $\underline\alpha$ and the Jensen inequality \cite{oksendal} to show the inequalities in (\ref{convexity}). Let us now show that condition (iii) in Definition
\ref{APSR} holds. Consider any \mbox{$\upsilon_{\tau}\in {U}_{\tau}$}. Choose an input \mbox{$u_{\params}\in U_{\params}$} satisfying
\begin{equation}
\Vert \upsilon_{\tau}-u_{\params}\Vert_{\infty}=\Vert \upsilon_{\tau}(0)-u_{\params}(0)\Vert\leq\mu.\label{b01}
\end{equation}
Note that the existence of such $u_\params$ is guaranteed by $\mathsf{U}$ being a finite union of boxes and by 
the inequality $\mu\leq\boxspan(\mathsf{U})$ which guarantees that $\mathsf{U}\subseteq\bigcup_{p\in[\mathsf{U}]_{\mu}}\mathcal{B}_{{\mu}}(p)$. Consider the transition \mbox{$x_{\tau}\rTo^{\upsilon_{\tau}}_{\tau} x'_{\tau}=\xi_{x_{\tau}\upsilon_{\tau}}(\tau)$} in $S_{\tau}(\Sigma)$. Since $V$ is a \mbox{$\delta$-ISS-M$_q$} Lyapunov function for $\Sigma$, in light of \eqref{V-bound} as well as \eqref{b01}, we have
\begin{align}\label{b02}
\mathbb{E}\left[V(x'_{\tau},\xi_{x_{\params}u_{\params}}(\tau))\right] &\leq \EE\left[V(x_\tau,x_q)\right] \mathsf{e}^{-\kappa\tau} + \frac{1}{\mathsf{e}\kappa} \rho(\|\upsilon_{\tau}-u_{\params}\|_\infty)\leq \underline\alpha\left(\varepsilon^q\right) \mathsf{e}^{-\kappa\tau} + \frac{1}{\mathsf{e}\kappa}\rho(\mu).
\end{align}
Since \mbox{$\R^n\subseteq\bigcup_{p\in[\mathbb{R}^n]_{\eta}}\mathcal{B}_{\eta}(p)$}, there exists a \mbox{$x'_{\params}\in{X}_{\params}$} such that 
\begin{equation}
\left\Vert\ol{\xi}_{x_{\params}u_{\params}}(\tau)-x'_{\params}\right\Vert\leq\eta, \label{b04}
\end{equation}
which, by the definition of $S_\params(\Sigma)$, implies the existence of $x_{\params}\rTo^{u_{\params}}_{\params}x'_{\params}$ in $S_{\params}(\Sigma)$. 
Using Lemmas \ref{lemma3} or \ref{lem:moment est}, the concavity of $\widehat\gamma$, the Jensen inequality \cite{oksendal}, the inequalities (\ref{supplement}), (\ref{bisim_cond}), (\ref{b02}), (\ref{b04}), and triangle inequality, we obtain
\begin{align*}
\mathbb{E}\left[V(x'_{\tau},x'_{\params})\right]&=\mathbb{E}\left[V(x'_\tau,\xi_{x_{\params}u_{\params}}(\tau))+V(x'_{\tau},x'_\params)-V(x'_\tau,\xi_{x_{\params}u_{\params}}(\tau))\right]\\ \notag
&=  \mathbb{E}\left[V(x'_{\tau},\xi_{x_{\params}u_{\params}}(\tau))\right]+\mathbb{E}\left[V(x'_{\tau},x'_\params)-V(x'_\tau,\xi_{x_{\params}u_{\params}}(\tau))\right]\\\notag&\leq\underline\alpha\left(\varepsilon^q\right)\mathsf{e}^{-\kappa\tau}+\frac{1}{\mathsf{e}\kappa}\rho(\mu)+\mathbb{E}\left[\widehat\gamma\left(\left\Vert\xi_{x_{\params}u_{\params}}(\tau)-x'_{\params}\right\Vert\right)\right]\\\notag
&\leq\underline\alpha\left(\varepsilon^q\right)\mathsf{e}^{-\kappa\tau}+\frac{1}{\mathsf{e}\kappa}\rho(\mu)+\widehat\gamma\left(\mathbb{E}\left[\left\Vert\xi_{x_{\params}u_{\params}}(\tau)-\ol{\xi}_{x_{\params}u_{\params}}(\tau)+\ol{\xi}_{x_{\params}u_{\params}}(\tau)-x'_{\params}\right\Vert\right]\right)\\\notag
&\leq\underline\alpha\left(\varepsilon^q\right)\mathsf{e}^{-\kappa\tau}+\frac{1}{\mathsf{e}\kappa}\rho(\mu)+\widehat\gamma\left(\mathbb{E}\left[\left\Vert\xi_{x_{\params}u_{\params}}(\tau)-\ol{\xi}_{x_{\params}u_{\params}}(\tau)\right\Vert\right]+\left\Vert\ol{\xi}_{x_{\params}u_{\params}}(\tau)-x'_{\params}\right\Vert\right)\\\notag
&\leq\underline\alpha\left(\varepsilon^q\right)\mathsf{e}^{-\kappa\tau}+\frac{1}{\mathsf{e}\kappa}\rho(\mu)+\widehat\gamma\left(\left(h(\sigma,\tau)\right)^{\frac{1}{q}}+\eta\right)\leq\underline\alpha\left(\varepsilon^q\right).
\end{align*}
Therefore, we conclude that \mbox{$\left(x'_{\tau},x'_{\params}\right)\in{R}$} and that condition (iii) in Definition \ref{APSR} holds. 

Now we prove \mbox{$S_{\params}(\Sigma)\preceq^{\varepsilon}_{\mathcal{S}}S_{\tau}(\Sigma)$} implying that $R^{-1}$ is a suitable
$\varepsilon$-approximate simulation relation from $S_{\params}(\Sigma)$ to $S_{\tau}(\Sigma)$. 
Consider the relation \mbox{$R\subseteq X_{\tau}\times X_{\params}$}, defined in the first part of the proof. For every \mbox{$x_{\params0}\in{X}_{\params0}$}, by choosing \mbox{$x_{\tau0}=x_{\params0}$}, we have $V\left(x_{\tau0},x_{\params0}\right)=0$ and \mbox{$\left(x_{\tau0},x_{\params0}\right)\in{R}$} and condition (i) in Definition \ref{APSR} is satisfied. Now consider any \mbox{$\left(x_{\tau},x_{\params}\right)\in R$}. Condition (ii) in Definition \ref{APSR} is satisfied because
\begin{equation} 
\label{convexity1}
\left(\mathbb{E}\left[\Vert x_{\tau}-x_{\params}\Vert^q\right]\right)^{\frac{1}{q}}\leq\left(\underline\alpha^{-1}\left(\mathbb{E}\left[V(x_{\tau},x_{\params})\right]\right)\right)^{\frac{1}{q}}\leq\varepsilon.
\end{equation}
We used the convexity assumption of $\underline\alpha$ and the Jensen inequality \cite{oksendal} to show the inequalities in (\ref{convexity1}). Let us now show that condition (iii) in Definition \ref{APSR} holds. 
Consider any \mbox{$u_{\params}\in U_{\params}$}. Choose the input \mbox{$\upsilon_{\tau}=u_\params$} and consider \mbox{$x'_{\tau}=\xi_{x_{\tau}\upsilon_{\tau}}(\tau)$ in $S_{\tau}(\Sigma)$}.
Since $V$ is a \mbox{$\delta$-ISS-M$_q$} Lyapunov function for $\Sigma$, one obtains:
\begin{equation}
\mathbb{E}\left[V(x'_{\tau},\xi_{x_{\params}u_{\params}}(\tau))\right]\leq \mathsf{e}^{-\kappa\tau}\EE[V\left(x_{\tau},x_\params\right)]\leq\mathsf{e}^{-\kappa\tau}\underline\alpha\left(\varepsilon^q\right).\label{b03}%
\end{equation}

Using Lemmas \ref{lemma3} or \ref{lem:moment est}, the definition of $S_\params(\Sigma)$, the concavity of $\widehat\gamma$, the Jensen inequality \cite{oksendal}, the inequalities (\ref{supplement}), (\ref{bisim_cond}), (\ref{b03}), and the triangle inequality, we obtain
\begin{align*}\nonumber
\mathbb{E}\left[V\left(x'_{\tau},x'_{\params}\right)\right]&=\mathbb{E}\left[V\left(x'_{\tau},\xi_{x_{\params}u_{\params}}(\tau)\right)+V\left(x'_{\tau},x'_{\params}\right)-V\left(x'_\tau,\xi_{x_{\params}u_{\params}}(\tau)\right)\right]\\\notag &= \mathbb{E}\left[V(x'_{\tau},\xi_{x_{\params}u_{\params}}(\tau))\right]+\mathbb{E}\left[V(x'_{\tau},x'_{\params})-V(x'_\tau,\xi_{x_{\params}u_{\params}}(\tau))\right]\\\notag&\leq\mathsf{e}^{-\kappa\tau}\underline\alpha\left(\varepsilon^q\right)+\mathbb{E}\left[\widehat\gamma\left(\left\Vert\xi_{x_{\params}u_{\params}}(\tau)-x'_{\params}\right\Vert\right)\right]\\\notag&\leq\mathsf{e}^{-\kappa\tau}\underline\alpha\left(\varepsilon^q\right)+\widehat\gamma\left(\mathbb{E}\left[\left\Vert\xi_{x_{\params}u_{\params}}(\tau)-\ol{\xi}_{x_{\params}u_{\params}}(\tau)+\ol{\xi}_{x_{\params}u_{\params}}(\tau)-x'_{\params}\right\Vert\right]\right)\\\notag&\leq\mathsf{e}^{-\kappa\tau}\underline\alpha\left(\varepsilon^q\right)+\widehat\gamma\left(\mathbb{E}\left[\left\Vert\xi_{x_{\params}u_{\params}}(\tau)-\ol{\xi}_{x_{\params}u_{\params}}(\tau)\right\Vert\right]+\left\Vert\ol{\xi}_{x_{\params}u_{\params}}(\tau)-x'_{\params}\right\Vert\right)\\\notag&\leq\mathsf{e}^{-\kappa\tau}\underline\alpha\left(\varepsilon^q\right)+\widehat\gamma\left(\left(h(\sigma,\tau)\right)^{\frac{1}{q}}+\eta\right)\leq\underline\alpha\left(\varepsilon^q\right).
\end{align*}
Therefore, we conclude that \mbox{$(x'_{\tau},x'_{\params})\in{R}$} and condition (iii) in Definition \ref{APSR} holds. 
\end{proof}

Note that the results in \cite{girard2}, as in the following corollary, are fully recovered by the statement in Theorem \ref{main_theorem} if the stochastic control system $\Sigma$ is not affected by any noise, implying that $h(\sigma,\tau)$ is identically zero and that the $\delta$-ISS-M$_q$ property reduces to the $\delta$-ISS property.

\begin{corollary}\label{theorem1}
Let $\Sigma=(\mathbb{R}^{n},\mathsf{U},\mathcal{U}_\tau,f,0_{n\times{p}})$ be a $\delta$-ISS control system admitting a $\delta$-ISS Lyapunov function $V$. For any $\varepsilon\in\R^+$, and any triple $\mathsf{q}=(\tau,\eta,\mu)$ of quantization parameters satisfying $\mu\leq\text{span}(\mathsf{U})$ and
\begin{align}\label{bisim_cond10}
\overline\alpha\left(\eta\right)\leq\underline\alpha(\varepsilon),\\\label{bisim_cond0}
\mathsf{e}^{-\kappa\tau}\underline\alpha(\varepsilon)+\frac{1}{\mathsf{e}\kappa}\rho(\mu)+\widehat\gamma\left(\eta\right)\leq\underline\alpha(\varepsilon),
\end{align}
one obtains \mbox{$S_{\mathsf{q}}(\Sigma)\cong_{\mathcal{S}}^{\varepsilon}S_{\tau}(\Sigma)$}.
\end{corollary} 

The next main theorem provides a result that is similar  to that in Theorem \ref{main_theorem}, 
which is however not obtained by explicit use of $\delta$-ISS-M$_q$ Lyapunov functions, 
but by using functions $\beta$ and $\gamma$ as in (\ref{delta_PISS}). 

\begin{theorem}\label{main_theorem2}
Consider a $\delta$-ISS-M$_q$ stochastic control system $\Sigma=(\mathbb{R}^{n},\mathsf{U},\mathcal{U}_\tau,f,\sigma)$, satisfying (\ref{mismatch1}). For any $\varepsilon\in\R^+$, and any triple $\mathsf{q}=(\tau,\eta,\mu)$ of quantization parameters satisfying $\mu\leq\boxspan(\mathsf{U})$ and
\begin{align}\label{bisim_cond2}
\left(\beta\left(\varepsilon^q,\tau\right)+\gamma(\mu)\right)^{\frac{1}{q}}+\left(h(\sigma,\tau)\right)^{\frac{1}{q}}+\eta\leq\varepsilon,
\end{align}
we have \mbox{$S_{\mathsf{q}}(\Sigma)\cong_{\mathcal{S}}^{\varepsilon}S_{\tau}(\Sigma)$}.
\end{theorem} 

It can be readily seen that when we are interested in the dynamics of $\Sigma$, initialized on a compact $D\subset\R^n$ of the form of finite union of boxes and for a given precision $\varepsilon$, 
there always exists a sufficiently large value of $\tau$ and small values of $\eta$ and $\mu$ such that $\eta\leq\boxspan(D)$ and the condition in (\ref{bisim_cond2}) are satisfied. 
However, unlike the result in Theorem \ref{main_theorem}, 
notice that here for a given fixed sampling time $\tau$, 
one may not find any values of $\eta$ and $\mu$ satisfying (\ref{bisim_cond2}) because the quantity $\left(\beta\left(\varepsilon^q,\tau\right)\right)^{\frac{1}{q}}$ may be larger than $\varepsilon$.  
As long as there exists a triple $\mathsf{q}$, satisfying (\ref{bisim_cond2}), the lower bound of precision $\varepsilon$ can be computed by solving the following inequality with respect to $\varepsilon$: $\varepsilon-\left(\beta\left(\varepsilon^q,\tau\right)\right)^{\frac{1}{q}}>\left(h(\sigma,\tau)\right)^{\frac{1}{q}}$.
In this case, one can easily verify that the lower bound on $\varepsilon$ goes to zero as $\tau$ goes to infinity, 
as $\tau$ goes to zero (only if $\left(\beta\left(\varepsilon^q,0\right)\right)^{\frac{1}{q}}\leq\varepsilon$), 
or as $Z \ra 0$, where $Z$ is the Lipschitz constant introduced in Definition \ref{Def_control_sys}. 
The symbolic model $S_\params(\Sigma)$, computed by using the quantization parameters $\mathsf{q}$ provided in Theorem \ref{main_theorem2} whenever existing,  is likely to have fewer states than the model computed by using the quantization parameters provided in Theorem \ref{main_theorem}, 
or may provide a better lower bound on $\varepsilon$ for a fixed sampling time $\tau$. 
We refer the interested reader to the case studies section for a detailed comparison between the results of Theorems \ref{main_theorem} and \ref{main_theorem2} on some practical examples. 

\begin{proof}
We start by proving \mbox{$S_{\tau}(\Sigma)\preceq^{\varepsilon}_\mathcal{S}S_{\params}(\Sigma)$}.
Consider the relation $R\subseteq X_{\tau}\times X_{\params}$ defined by 
{$\left(x_{\tau},x_{\params}\right)\in R$}
if and only if 
{$\left(\mathbb{E}\left[\left\Vert H_{\tau}(x_{\tau})-H_{\params}(x_{\params})\right\Vert^q\right]\right)^{\frac{1}{q}}=\left(\mathbb{E}\left[\left\Vert x_{\tau}-x_{\params}\right\Vert^q\right]\right)^{\frac{1}{q}}\leq\varepsilon$}. 
Since \mbox{$X_{\tau 0}\subseteq\bigcup_{p\in[\mathbb{R}^n]_{\eta}}\mathcal{B}_{\eta}(p)$}, 
for every $x_{\tau 0}\in{X_{\tau 0}}$ there always exists \mbox{$x_{\params 0}\in{X}_{\params 0}$} such that $\Vert{x_{\tau0}}-x_{\params0}\Vert\leq\eta$.
Then,
\begin{equation}\nonumber
\left(\mathbb{E}\left[\left\Vert{x_{\tau0}}-x_{\params0}\right\Vert^q\right]\right)^{\frac{1}{q}}=\left(\left\Vert{x_{\tau0}}-x_{\params0}\right\Vert^q\right)^{\frac{1}{q}}\leq\eta\leq\varepsilon,
\end{equation}
because of (\ref{bisim_cond2}).
Hence, \mbox{$\left(x_{\tau0},x_{\params0}\right)\in{R}$} and condition (i) in Definition \ref{APSR} is satisfied. Now consider any \mbox{$\left(x_{\tau},x_{\params}\right)\in R$}. Condition (ii) in Definition \ref{APSR} is satisfied by the definition of $R$. Let us now show that condition (iii) in Definition
\ref{APSR} holds. Consider any \mbox{$\upsilon_{\tau}\in {U}_{\tau}$}. Choose an input \mbox{$u_{\params}\in U_{\params}$} satisfying
\begin{equation}
\Vert \upsilon_{\tau}-u_{\params}\Vert_{\infty}=\Vert \upsilon_{\tau}(0)-u_{\params}(0)\Vert\leq\mu.\label{b10}
\end{equation}
Note that the existence of such $u_\params$ is guaranteed by $\mathsf{U}$ being a finite union of boxes and by 
the inequality $\mu\leq\boxspan(\mathsf{U})$ which guarantees that $\mathsf{U}\subseteq\bigcup_{p\in[\mathsf{U}]_{\mu}}\mathcal{B}_{{\mu}}(p)$. Consider the transition \mbox{$x_{\tau}\rTo^{\upsilon_{\tau}}_{\tau} x'_{\tau}=\xi_{x_{\tau}\upsilon_{\tau}}(\tau)$} in $S_{\tau}(\Sigma)$. It follows from the $\delta$-ISS-M$_q$ assumption on $\Sigma$ and (\ref{b10}) that: 
\begin{align}\label{b20}
\mathbb{E}\left[\left\Vert x'_{\tau}-\xi_{x_{\params}u_{\params}}(\tau)\right\Vert^q\right] &\leq \beta\left(\EE\left[\left\Vert x_\tau-x_q\right\Vert^q\right],\tau\right) + \gamma(\|\upsilon_{\tau}-u_{\params}\|_\infty)\leq \beta\left(\varepsilon^q,\tau\right) + \gamma(\mu).
\end{align}
Since \mbox{$\R^n\subseteq\bigcup_{p\in[\mathbb{R}^n]_{\eta}}\mathcal{B}_{\eta}(p)$}, there exists \mbox{$x'_{\params}\in{X}_{\params}$} such that 
\begin{equation}
\left\Vert\ol{\xi}_{x_{\params}u_{\params}}(\tau)-x'_{\params}\right\Vert\leq\eta, \label{b40}
\end{equation}
which, by the definition of $S_\params(\Sigma)$, implies the existence of $x_{\params}\rTo^{u_{\params}}_{\params}x'_{\params}$ in $S_{\params}(\Sigma)$. 
Using Lemmas \ref{lemma3} or \ref{lem:moment est}, (\ref{bisim_cond2}), (\ref{b20}), (\ref{b40}), and triangle inequality, we obtain
\begin{align*}
\left(\mathbb{E}\left[\left\Vert x'_{\tau}-x'_{\params}\right\Vert^q\right]\right)^{\frac{1}{q}}&=\left(\mathbb{E}\left[\left\Vert x'_\tau-\xi_{x_{\params}u_{\params}}(\tau)+\xi_{x_{\params}u_{\params}}(\tau)-\ol{\xi}_{x_{\params}u_{\params}}(\tau)+\ol{\xi}_{x_{\params}u_{\params}}(\tau)-x'_\params\right\Vert^q\right]\right)^{\frac{1}{q}}\\ \notag
&\leq  \left(\mathbb{E}\left[\left\Vert x'_{\tau}-\xi_{x_{\params}u_{\params}}(\tau)\right\Vert^q\right]\right)^{\frac{1}{q}}+\left(\mathbb{E}\left[\left\Vert\xi_{x_{\params}u_{\params}}(\tau)-\ol{\xi}_{x_{\params}u_{\params}}(\tau)\right\Vert^q\right]\right)^{\frac{1}{q}}+\left(\mathbb{E}\left[\left\Vert\ol{\xi}_{x_{\params}u_{\params}}(\tau)-x'_\params\right\Vert^q\right]\right)^{\frac{1}{q}}\\\notag&\leq\left(\beta\left(\varepsilon^q,\tau\right) + \gamma(\mu)\right)^{\frac{1}{q}}+\left(h(\sigma,\tau)\right)^{\frac{1}{q}}+\eta\leq\varepsilon.
\end{align*}
Therefore, we conclude that \mbox{$\left(x'_{\tau},x'_{\params}\right)\in{R}$} and that condition (iii) in Definition \ref{APSR} holds. 

Now we prove \mbox{$S_{\params}(\Sigma)\preceq^{\varepsilon}_{\mathcal{S}}S_{\tau}(\Sigma)$} implying that $R^{-1}$ is a suitable
$\varepsilon$-approximate simulation relation from $S_{\params}(\Sigma)$ to $S_{\tau}(\Sigma)$. 
Consider the relation \mbox{$R\subseteq X_{\tau}\times X_{\params}$}, defined in the first part of the proof. For every \mbox{$x_{\params0}\in{X}_{\params0}$}, by choosing \mbox{$x_{\tau0}=x_{\params0}$}, we have $\left\Vert x_{\tau0}-x_{\params0}\right\Vert^q=0$ and \mbox{$\left(x_{\tau0},x_{\params0}\right)\in{R}$} and condition (i) in Definition \ref{APSR} is satisfied. Now consider any \mbox{$\left(x_{\tau},x_{\params}\right)\in R$}. Condition (ii) in Definition \ref{APSR} is satisfied by the definition of $R$. Let us now show that condition (iii) in Definition \ref{APSR} holds. 
Consider any \mbox{$u_{\params}\in U_{\params}$}. Choose the input \mbox{$\upsilon_{\tau}=u_\params$} and consider \mbox{$x'_{\tau}=\xi_{x_{\tau}\upsilon_{\tau}}(\tau)$ in $S_{\tau}(\Sigma)$}.
Since $\Sigma$ is \mbox{$\delta$-ISS-M$_q$}, one obtains:
\begin{equation}
\mathbb{E}\left[\left\Vert x'_{\tau}-\xi_{x_{\params}u_{\params}}(\tau)\right\Vert^q\right]\leq \beta\left(\EE[\left\Vert x_{\tau}-x_\params\right\Vert^q],\tau\right)\leq\beta\left(\varepsilon^q,\tau\right).\label{b30}%
\end{equation}

Using Lemmas \ref{lemma3} or \ref{lem:moment est}, the definition of $S_\params(\Sigma)$, (\ref{bisim_cond2}), (\ref{b30}), and the triangle inequality, we obtain
\begin{align*}
\left(\mathbb{E}\left[\left\Vert x'_{\tau}-x'_{\params}\right\Vert^q\right]\right)^{\frac{1}{q}}&=\left(\mathbb{E}\left[\left\Vert x'_\tau-\xi_{x_{\params}u_{\params}}(\tau)+\xi_{x_{\params}u_{\params}}(\tau)-\ol{\xi}_{x_{\params}u_{\params}}(\tau)+\ol{\xi}_{x_{\params}u_{\params}}(\tau)-x'_\params\right\Vert^q\right]\right)^{\frac{1}{q}}\\ \notag
&\leq  \left(\mathbb{E}\left[\left\Vert x'_{\tau}-\xi_{x_{\params}u_{\params}}(\tau)\right\Vert^q\right]\right)^{\frac{1}{q}}+\left(\mathbb{E}\left[\left\Vert\xi_{x_{\params}u_{\params}}(\tau)-\ol{\xi}_{x_{\params}u_{\params}}(\tau)\right\Vert^q\right]\right)^{\frac{1}{q}}+\left(\mathbb{E}\left[\left\Vert\ol{\xi}_{x_{\params}u_{\params}}(\tau)-x'_\params\right\Vert^q\right]\right)^{\frac{1}{q}}\\\notag&\leq\left(\beta\left(\varepsilon^q,\tau\right)\right)^{\frac{1}{q}}+\left(h(\sigma,\tau)\right)^{\frac{1}{q}}+\eta\leq\varepsilon.
\end{align*}
Therefore, we conclude that \mbox{$(x'_{\tau},x'_{\params})\in{R}$} and condition (iii) in Definition \ref{APSR} holds. 
\end{proof}

Note that the results in \cite{pola}, as in the following corollary, are fully recovered by the results in Theorem \ref{main_theorem2} if the stochastic control system $\Sigma$ is not affected by any noise, implying that $h(\sigma,\tau)$ is identically zero and that the $\delta$-ISS-M$_q$ property becomes the $\delta$-ISS property.

\begin{corollary}\label{theorem10}
Let $\Sigma=(\mathbb{R}^{n},\mathsf{U},\mathcal{U}_\tau,f,0_{n\times{p}})$ be a $\delta$-ISS control system (i.e. satisfying \eqref{delta_ISS}). For any $\varepsilon\in\R^+$, and any triple $\mathsf{q}=(\tau,\eta,\mu)$ of quantization parameters satisfying $\mu\leq\text{span}(\mathsf{U})$ and
\begin{align}\label{bisim_cond200}
\beta(\varepsilon,\tau)+\gamma(\mu)+\eta\leq\varepsilon,
\end{align}
we have \mbox{$S_{\mathsf{q}}(\Sigma)\cong_{\mathcal{S}}^{\varepsilon}S_{\tau}(\Sigma)$}.
\end{corollary} 

Symbolic models can be easily model checked or employed towards controller synthesis. 
It is of interest to understand how abstract controllers can be used over the concrete models. 
The next proposition elucidates how a controller $S_{cont}$ synthesized to solve a simulation game over $S_\params(\Sigma)$ can be refined to a controller for $S_\tau(\Sigma)$. 
A detailed description of the feedback composition (denoted by $\parallel$) and of its properties for metric systems can be found in \cite{paulo}. 

\begin{proposition}\cite[Proposition 11.10]{paulo}
Consider a stochastic control system $\Sigma$, and a specification described by a deterministic system $S_{spec}=\left(X_{spec},X_{spec0},U_{spec},\rTo_{spec},Y_{spec},H_{spec}\right)$, where $X_{spec}$ is a finite subset of $\R^n$, $X_{spec0}\subseteq X_{spec}$, $U_{spec}=\left\{u_{spec}\right\}$, $\rTo_{spec}\subseteq X_{spec}\times U_{spec}\times X_{spec}$, $Y_{spec}$ is the set of all $\R^n$-valued random variables defined on the probability space $(\Omega,\sigalg,\PP)$, and $H_{spec}=\imath:X_{spec}\hookrightarrow Y_{spec}$. Assume that $S_{\tau}(\Sigma)\cong^{\varepsilon}_{\mathcal{S}}S_{\params}(\Sigma)$ and $S_{cont}$ is synthesized to solve exactly a simulation game for $S_\params(\Sigma)$ and a specification $S_{spec}$: $S_{cont}\parallel S_{\params}(\Sigma)\preceq^{0}_{\mathcal{S}}S_{spec}$ (resp. $S_{cont}\parallel S_{\params}(\Sigma)\cong^{0}_{\mathcal{S}}S_{spec}$). Then, using $S'_{cont}=S_{cont}\parallel S_{\params}(\Sigma)$ as a controller for $S_\tau(\Sigma)$, we obtain: $S'_{cont}\parallel S_{\tau}(\Sigma)\preceq^{\varepsilon}_{\mathcal{S}}S_{spec}$ (resp. $S'_{cont}\parallel S_{\tau}(\Sigma)\cong^{\varepsilon}_{\mathcal{S}}S_{spec}$).
\end{proposition}

\begin{remark}\label{remark5}
Although we assume that the set $\mathsf{U}$ is infinite, 
Theorems \ref{main_theorem} (resp. Corollary \ref{theorem1}) and \ref{main_theorem2} (resp. Corollary \ref{theorem10}) still hold when the set $\mathsf{U}$ is finite, 
with the following modifications. 
First, the system $\Sigma$ is required to satisfy the property (\ref{delta_PISS}) for $\upsilon=\upsilon'$. 
Second, take $U_\params=\mathsf{U}$ 
in the definition of $S_\params(\Sigma)$.
Finally, in the conditions (\ref{bisim_cond}) (resp. \eqref{bisim_cond0}) and (\ref{bisim_cond2}) (resp. \eqref{bisim_cond200}) set $\mu=0$. \hfill$\Box$
\end{remark}

\begin{remark}
If stochastic control system $\Sigma$ is not $\delta$-ISS-M$_q$, one can use the results in \cite{majid7}, providing symbolic models that are only sufficient in the sense that the refinement of any controller synthesized for the symbolic model enforces the same desired specifications on the original stochastic control system (in the sense of moments). 
However, they can no longer guarantee, as it was the case in this paper, that the existence of a controller for the original stochastic control system leads to the existence of a controller for the symbolic model.  \hfill$\Box$
\end{remark}

\subsection{Relationship with notions of probabilistic abstractions in the literature}

In the remainder of this section we relate the (approximate) equivalence relations introduced in this work with known probabilistic concepts in the literature. 

The approximate bisimulation notion in Definition \ref{APSR} is structurally different than the probabilistic version discussed for finite state, 
\emph{discrete-time}  labeled Markov chains in \cite{DLT08}, which is also extended to continuous-space processes as in \cite{DP02} and hinges on the one-step difference between transition probabilities over sets covering the state space.    
The notion in this work can be instead related to the approximate probabilistic bisimulation notion discussed in \cite{julius1},  
which lower bounds the probability that the Euclidean distance between any trace of abstract model and the corresponding one of the concrete model and vice versa remains close:    
both notions hinge on distances over trajectories, rather than over transition probabilities as in \cite{DP02,DLT08}.  

We make the above statement more precise with the following results, 
which do not require any more the assumptions that $f(0_n,0_m)=0_n$ and $\sigma(0_n)=0_{n\times{p}}$.  

\begin{theorem}\label{lemma5}
Let $\Sigma=(\mathbb{R}^{n},\mathsf{U},\mathcal{U}_\tau,f,\sigma)$ be a stochastic control system. Assume that there exists a stochastic bisimulation function $\phi:\R^n\times\R^n\ra\R_0^+$, as defined in \cite{julius1}, between $\Sigma$ and its corresponding non-probabilistic control system $\ol\Sigma=(\mathbb{R}^{n},\mathsf{U},\mathcal{U}_\tau,f,0_{n\times{p}})$, and assume that $\ol\Sigma$ is $\delta$-ISS. For any $\varepsilon\in\R^+$, any triple $\params=(\tau,\eta,\mu)$ of quantization parameters, satisfying $\mu\leq\boxspan(\mathsf{U})$ and \eqref{bisim_cond200} (or \eqref{bisim_cond10} and \eqref{bisim_cond0}), and for any state run $\xi_{x_{\tau0}\upsilon}$ of $S_\tau(\Sigma)$, there exists a state run $\mathsf{X}_{x_{\params0}u}$ of $S_\params(\Sigma)$, and vice versa, such that the following relation holds:
\begin{equation}\label{simultaneous}
\PP\left\{\sup_{k\in{\N_0}}\left\Vert\xi_{x_{\tau0}\upsilon}(k\tau)-\mathsf{X}_{x_{\params0}u}(k)\right\Vert>\epsilon\,\,|\,\,x_{\tau0}\right\}\leq\frac{\sqrt{\phi(x_{\tau0},x_{\tau0})}+\varepsilon}{\epsilon}.
\end{equation}
\end{theorem}

\begin{proof}
Following the definition of stochastic bisimulation function in \cite{julius1}, for any input curve $\upsilon\in\mathcal{U}_\tau$, there exists an input curve $\ol\upsilon\in\mathcal{U}_\tau$ such that $\phi\left(\xi_{x_{\tau0}\upsilon},\ol\xi_{x_{\tau0}\ol\upsilon}\right)$ is a nonnegative supermartingale \cite[Appendix C]{oksendal}. 
Furthermore, using the chosen triple $\params=(\tau,\eta,\mu)$ of quantization parameters, we have: $S_\tau\left(\ol\Sigma\right)\cong_{\mathcal{S}}^{\varepsilon}S_{\mathsf{q}}(\Sigma)$, implying that for a state run $\ol\xi_{x_{\tau0}\ol\upsilon}$ of $S_\tau\left(\ol\Sigma\right)$,  there exists a state run $\mathsf{X}_{x_{\params0}u}$ of $S_\params(\Sigma)$ such that: 
$$\sup_{k\in{\N_0}}\left\Vert\ol\xi_{x_{\tau0}\ol\upsilon}(k\tau)-\mathsf{X}_{x_{\params0}u}(k)\right\Vert\leq\varepsilon.$$ 
The above statement follows by direct application of Definition \ref{APSR}. 
Since an increasing concave function of a supermartingale is a supermartingale \cite[Theorem 5.11]{yeh1995martingales}, we have the following chain of (in)equalities:
\begin{align}\nonumber
&\PP\left\{\sup_{k\in{\N_0}}\left\Vert\xi_{x_{\tau0}\upsilon}(k\tau)-\mathsf{X}_{x_{\params0}u}(k)\right\Vert>\epsilon\,\,|\,\,x_{\tau0}\right\}=\\\notag
&\PP\left\{\sup_{k\in{\N_0}}\left\Vert\xi_{x_{\tau0}\upsilon}(k\tau)-\ol\xi_{x_{\tau0}\ol\upsilon}(k\tau)+\ol\xi_{x_{\tau0}\ol\upsilon}(k\tau)-\mathsf{X}_{x_{\params0}u}(k)\right\Vert>\epsilon\,\,|\,\,x_{\tau0}\right\}\leq\\\notag
&\PP\left\{\sup_{k\in{\N_0}}\left\{\left\Vert\xi_{x_{\tau0}\upsilon}(k\tau)-\ol\xi_{x_{\tau0}\ol\upsilon}(k\tau)\right\Vert+\left\Vert\ol\xi_{x_{\tau0}\ol\upsilon}(k\tau)-\mathsf{X}_{x_{\params0}u}(k)\right\Vert\right\}>\epsilon\,\,|\,\,x_{\tau0}\right\}\leq\\\notag
&\PP\left\{\sup_{k\in{\N_0}}\left\{\left(\left\Vert\xi_{x_{\tau0}\upsilon}(k\tau)-\ol\xi_{x_{\tau0}\ol\upsilon}(k\tau)\right\Vert^2\right)^{\frac{1}{2}}\right\}+\varepsilon>\epsilon\,\,|\,\,x_{\tau0}\right\}\leq\\\notag
&\PP\left\{\sup_{k\in{\N_0}}\left\{\left(\phi\left(\xi_{x_{\tau0}\upsilon}(k\tau),\ol\xi_{x_{\tau0}\ol\upsilon}(k\tau)\right)\right)^{\frac{1}{2}}\right\}+\varepsilon>\epsilon\,\,|\,\,x_{\tau0}\right\}\leq\frac{\sqrt{\phi\left(x_{\tau0},x_{\tau0}\right)}+\varepsilon}{\epsilon}.
\end{align} 
In a similar way, we can prove that for any state run $\mathsf{X}_{x_{\params0}u}$ of $S_\params(\Sigma)$, there exists a state run $\xi_{x_{\tau0}\upsilon}$ of $S_\tau(\Sigma)$ such that the relation in (\ref{simultaneous}) holds.
\end{proof}

Theorem 3 in \cite{julius1} provides a similar result as in Theorem \ref{lemma5}, 
however it is limited to hold over two infinite systems rather than a finite and an infinite one as in this work. 
As explained in \cite{julius1}, in order to compute a stochastic bisimulation function, one requires to solve a game, 
which is computationally difficult even for linear stochastic control systems. 
On the other hand, if the set of input is a singleton (i.e., $\mathsf{U}=\{0_m\}$), 
which results in the verification of stochastic dynamical systems, then a stochastic bisimulation function $\phi$ can be efficiently computed by solving some LMIs, 
as explained in \cite[equations (39) and (40)]{julius1} for linear stochastic control systems. 
Alternatively, the stochastic contractivity of the model can be used for this goal \cite{abate}. 

\begin{remark} 
Let us further comment on the relationship between $\delta$-ISS-M$_q$ Lyapunov functions in this work and stochastic bisimulation ones in \cite{julius1}. 
Let $\Sigma=(\mathbb{R}^{n},\{0_m\},\mathcal{U}_\tau,f,\sigma)$ be a stochastic control system admitting a $\delta$-ISS-M$_q$ Lyapunov function $V$, and such that $f(0_n,0_m)=0_n$, $\sigma(0_n)=0_{n\times{p}}$, and $0_{n\times{n}}\preceq\partial_{\ol{x},\ol{x}}V(0_n,\ol{x})$ for all $\ol{x}\in\R^n$. It can be readily verified that the function $\phi(x,\ol{x})=V(x,0_n)+V(0_n,\ol{x})$, for all $x,\ol{x}\in\R^n$, is a stochastic bisimulation function 
between $\Sigma$ and the corresponding non-probabilistic control system $\ol\Sigma$, 
under the following modification in condition (i) in \cite[Definition 2]{julius1}: (i) $\phi(x,\ol{x})\geq\widetilde\alpha\left(\Vert x-\ol{x}\Vert\right)$, for some $\widetilde\alpha\in\mathcal{K}_\infty$ and all $x,\ol{x}\in\R^n$. Using this new condition on $\phi$, one can readily revise the relation \eqref{simultaneous} correspondingly. \hfill$\Box$
\end{remark}

The next theorem, which is computationally more tractable for stochastic control systems and does not require the existence of stochastic bisimulation functions as in \cite{julius1}, provides a similar result as in Theorem \ref{lemma5}, 
but applies only over a finite time horizon. 

\begin{theorem}\label{lemma6}
Let $\Sigma=(\mathbb{R}^{n},\mathsf{U},\mathcal{U}_\tau,f,\sigma)$ be a stochastic control system. Suppose there exists a $\delta$-ISS-M$_q$ Lyapunov function $V$ for $\Sigma$ with a form as in (\ref{V}) or as in Lemma \ref{lemma3}, and that $\Tr(\sigma(x)^T\partial_{x,x}V(x,y)\sigma(x))$ is uniformly upper-bounded by a constant $\alpha\ge0$, for all $x,y\in \R^n$. For any $\varepsilon\in\R^+$, any triple $\params=(\tau,\eta,\mu)$ of quantization parameters, satisfying $\mu\leq\boxspan(\mathsf{U})$ and \eqref{bisim_cond200} (or \eqref{bisim_cond10}, and \eqref{bisim_cond0}), and for any state run $\xi_{x_{\tau0}\upsilon}$ of $S_\tau(\Sigma)$, there exists a state run $\mathsf{X}_{x_{\params0}u}$ of $S_\params(\Sigma)$, and vice versa, such that the following relations hold:
\begin{align}\label{simultaneous1}
\PP\left\{\sup_{0\leq k\leq N}\left\Vert\xi_{x_{\tau0}\upsilon}(k\tau)-\mathsf{X}_{x_{\params0}u}(k)\right\Vert>\epsilon+\varepsilon\,\,|\,\,x_{\tau0}\right\}&\leq1-\mathsf{e}^{-\frac{\alpha N\tau}{2\ul\alpha\left(\epsilon^q\right)}},~\text{if}~\ul\alpha\left(\epsilon^q\right)\geq\frac{\alpha}{2\kappa},\\\label{simultaneous2}
\PP\left\{\sup_{0\leq k\leq N}\left\Vert\xi_{x_{\tau0}\upsilon}(k\tau)-\mathsf{X}_{x_{\params0}u}(k)\right\Vert>\epsilon+\varepsilon\,\,|\,\,x_{\tau0}\right\}&\leq\frac{\left(\mathsf{e}^{N\tau\kappa}-1\right)\alpha}{2\kappa\ul\alpha\left(\epsilon^q\right)\mathsf{e}^{N\tau\kappa}}, ~\text{if}~\ul\alpha\left(\epsilon^q\right)\leq\frac{\alpha}{2\kappa}.
\end{align}
\end{theorem}

\begin{proof}
As argued in the proof of Theorem \ref{lemma5}, for any state run $\ol\xi_{x_{\tau0}\upsilon}$ of $S_\tau\left(\ol\Sigma\right)$ (notice that here the non-probabilistic model is excited with the same input $\upsilon$ of $\Sigma$), there exists a state run $\mathsf{X}_{x_{\params0}u}$ of $S_\params(\Sigma)$ such that:$$\sup_{k\in{\N_0}}\left\Vert\ol\xi_{x_{\tau0}\upsilon}(k\tau)-\mathsf{X}_{x_{\params0}u}(k)\right\Vert\leq\varepsilon.$$ Hence, one obtains the following chain of (in)equalities:
\begin{align}\nonumber
&\PP\left\{\sup_{0\leq k\leq N}\left\Vert\xi_{x_{\tau0}\upsilon}(k\tau)-\mathsf{X}_{x_{\params0}u}(k)\right\Vert>\epsilon+\varepsilon\,\,|\,\,x_{\tau0}\right\}=\\\notag
&\PP\left\{\sup_{0\leq k\leq N}\left\Vert\xi_{x_{\tau0}\upsilon}(k\tau)-\ol\xi_{x_{\tau0}\upsilon}(k\tau)+\ol\xi_{x_{\tau0}\upsilon}(k\tau)-\mathsf{X}_{x_{\params0}u}(k)\right\Vert>\epsilon+\varepsilon\,\,|\,\,x_{\tau0}\right\}\leq\\\notag
&\PP\left\{\sup_{0\leq k\leq N}\left\{\left\Vert\xi_{x_{\tau0}\upsilon}(k\tau)-\ol\xi_{x_{\tau0}\upsilon}(k\tau)\right\Vert+\left\Vert\ol\xi_{x_{\tau0}\upsilon}(k\tau)-\mathsf{X}_{x_{\params0}u}(k)\right\Vert\right\}>\epsilon+\varepsilon\,\,|\,\,x_{\tau0}\right\}\leq\\\notag
&\PP\left\{\sup_{0\leq k\leq N}\left\{\ul\alpha\left(\left\Vert\xi_{x_{\tau0}\upsilon}(k\tau)-\ol\xi_{x_{\tau0}\upsilon}(k\tau)\right\Vert^q\right)\right\}>\ul\alpha\left(\epsilon^q\right)\,\,|\,\,x_{\tau0}\right\}\leq\\\label{final_probability}
&\PP\left\{\sup_{0\leq k\leq N}\left\{V\left(\xi_{x_{\tau0}\upsilon}(k\tau),\ol\xi_{x_{\tau0}\upsilon}(k\tau)\right)\right\}>\ul\alpha\left(\epsilon^q\right)\,\,|\,\,x_{\tau0}\right\}.
\end{align} 
As showed for the $\delta$-ISS-M$_q$ Lyapunov function $V$ in the proof of Lemma \ref{lemma3}, we have:
\begin{align}\nonumber
&\mathcal{L}^{\upsilon(t),\upsilon(t)}V\left(\xi_{x_{\tau0}\upsilon}(t),\ol\xi_{x_{\tau0}\upsilon}(t)\right)\\\notag&\leq-\kappa V\left(\xi_{x_{\tau0}\upsilon}(t),\ol\xi_{x_{\tau0}\upsilon}(t)\right)+\frac{1}{2}\Tr\left(\sigma(\xi_{x_{\tau0}\upsilon}(t))^T\partial_{x,x}V(\xi_{x_{\tau0}\upsilon}(t),\ol\xi_{x_{\tau0}\upsilon}(t))\sigma(\xi_{x_{\tau0}\upsilon}(t))\right)\\\label{bias}
&\leq-\kappa V\left(\xi_{x_{\tau0}\upsilon}(t),\ol\xi_{x_{\tau0}\upsilon}(t)\right)+\frac{\alpha}{2},
\end{align}
for any $\upsilon\in\mathcal{U}_\tau$ and any $x_{\tau0}\in\R^n$.
Using inequalities \eqref{final_probability}, \eqref{bias}, and Theorem 1 in \cite[Chapter III]{kushner}, one obtains the relations (\ref{simultaneous1}) and (\ref{simultaneous2}).
In a similar way, we can prove that for any state run $\mathsf{X}_{x_{q0}u}$ of $S_\params(\Sigma)$, there exists a state run $\xi_{x_{\tau0}\upsilon}$ of $S_\tau(\Sigma)$ such that the relations in (\ref{simultaneous1}) and (\ref{simultaneous2}) hold.
\end{proof}

As an alternative to the two results above, 
we now introduce a probabilistic approximate bisimulation relation between $S_\tau(\Sigma)$ and $S_\params(\Sigma)$ point-wise in time: 
this relation is sufficient to work with LTL specifications which satisfiability can be ascertained at single time instances,  
such as for next ($\bigcirc$) and eventually ($\diamondsuit$). 


\begin{definition}\label{NAPSR}
Consider two systems $S_{\tau}(\Sigma)=(X_{\tau},X_{\tau0},U_{\tau},\rTo_{\tau},Y_\tau,H_{\tau})$ and\\ $S_{\params}(\Sigma)=(X_{\params},X_{\params0},U_{\params},\rTo_{\params},Y_\params,H_{\params})$, and precisions $\epsilon\in\mathbb{R}^{+}$ and $\widehat\varepsilon\in[0,1]$. A relation \mbox{$R\subseteq X_{\tau}\times X_{\params}$} is said to be an $(\epsilon,\widehat\varepsilon)$-approximate simulation relation from $S_{\tau}$ to $S_{\params}$, if the following three conditions are satisfied:

\begin{itemize}
\item[(i)] for every $x_{\tau0}\in{X_{\tau0}}$, there exists $x_{\params0}\in{X_{\params0}}$ with $(x_{\tau0},x_{\params0})\in{R}$;
\item[(ii)] for every $(x_{\tau},x_{\params})\in R$ we have $\PP\left\{\left\Vert H_{\tau}(x_{\tau})-H_{\params}(x_{\params})\right\Vert\geq\epsilon\right\}\leq\widehat\varepsilon$;
\item[(iii)] for every $(x_{\tau},x_{\params})\in R$ we have that \mbox{$x_{\tau}\rTo_{\tau}^{\upsilon_\tau}x'_{\tau}$ in $S_\tau$} implies the existence of \mbox{$x_{\params}\rTo_{\params}^{u_\params}x'_{\params}$} in $S_\params$ satisfying $(x'_{\tau},x'_{\params})\in R$.
\end{itemize}  

A relation $R\subseteq X_{\tau}\times X_{\params}$ is said to be an $(\epsilon,\widehat\varepsilon)$-approximate bisimulation relation between $S_{\tau}$ and $S_{\params}$
if $R$ is an $(\epsilon,\widehat\varepsilon)$-approximate simulation relation from $S_\tau$ to $S_\params$ and
$R^{-1}$ is an $(\epsilon,\widehat\varepsilon)$-approximate simulation relation from $S_\params$ to $S_\tau$.

System $S_{\tau}$ is $(\epsilon,\widehat\varepsilon)$-approximately simulated by $S_{\params}$, or $S_\params$ $(\epsilon,\widehat\varepsilon)$-approximately simulates $S_\tau$, 
denoted by \mbox{$S_{\tau}\preceq_{\mathcal{S}}^{(\epsilon,\widehat\varepsilon)}S_{\params}$}, if there exists
an $(\epsilon,\widehat\varepsilon)$-approximate simulation relation from $S_{\tau}$ to $S_{\params}$.
System $S_{\tau}$ is $(\epsilon,\widehat\varepsilon)$-approximate bisimilar to $S_{\params}$, denoted by \mbox{$S_{\tau}\cong_{\mathcal{S}}^{(\epsilon,\widehat\varepsilon)}S_{\params}$}, if there exists
an $(\epsilon,\widehat\varepsilon)$-approximate bisimulation relation $R$ between $S_{\tau}$ and $S_{\params}$.
\end{definition}

We show next the existence of symbolic models with respect to the notion of approximate bisimulation relation in Definition \ref{NAPSR}. 

\begin{theorem}\label{lemma2}
Let $\Sigma=(\mathbb{R}^{n},\mathsf{U},\mathcal{U}_\tau,f,\sigma)$ be a stochastic control system. For any $\epsilon\in\R^+$ and $\widehat\varepsilon\in[0,1]$, we have \mbox{$S_{\tau}(\Sigma)\cong_{\mathcal{S}}^{(\epsilon,\widehat\varepsilon)}S_{\params}(\Sigma)$} if \mbox{$S_{\tau}(\Sigma)\cong_{\mathcal{S}}^{\varepsilon}S_{\params}(\Sigma)$} for $\varepsilon=\epsilon\widehat\varepsilon$.
\end{theorem}

\begin{proof}
The proof is a simple consequence of Theorems \ref{main_theorem} or \ref{main_theorem2} and Markov inequality \cite{oksendal}, used as the following:
\begin{align}\nonumber
\PP\left\{\left\Vert H_{\tau}(x_{\tau})-H_{\params}(x_{\params})\right\Vert\geq\epsilon\right\}&\leq\frac{\EE\left[\left\Vert H_{\tau}(x_{\tau})-H_{\params}(x_{\params})\right\Vert\right]}{\epsilon}
\leq\frac{\left(\EE\left[\left\Vert H_{\tau}(x_{\tau})-H_{\params}(x_{\params})\right\Vert^q\right]\right)^{\frac{1}{q}}}{\epsilon}\leq\frac{\varepsilon}{\epsilon}=\widehat\varepsilon.
\end{align}
\end{proof}

\section{Case Studies}
We now experimentally demonstrate the effectiveness of the discussed results. 
In the examples below, the computation of the abstract systems $S_\params(\Sigma)$ have been implemented by the software tool \textsf{Pessoa}~\cite{pessoa} on a laptop with CPU $2$GHz Intel Core i$7$. We have assumed that the control inputs are piecewise constant of duration $\tau$ and that $\mathcal{U}_\tau$ is finite and contains curves taking values in $[\mathsf{U}]_{\mu}$. Hence, as explained in Remark \ref{remark5}, $\mu=0$ is to be used in the conditions (\ref{bisim_cond}) and (\ref{bisim_cond2}). 
The controllers enforcing the specifications of interest have been found by standard algorithms from game theory \cite{MalerPnueliSifakis95,Thomas95}, 
as implemented in \textsf{Pessoa}. 
In both examples, the terms $W_t^i, i=1,2,$ denote the standard Brownian motion terms.

\subsection{Nonlinear model, $2{\textsf{nd}}$ moment, ``sequential target tracking'' property}
Consider the nonlinear model of a pendulum on a cart borrowed from \cite{pola4}, which is now affected by noise. The model is described by:
\begin{equation}\label{example}
    \Sigma:\left\{
                \begin{array}{ll}
                  \diff{\xi}_1=\xi_2\diff{t}+0.03\xi_1\diff{W^1_t},\\ 
                 \diff{\xi}_2=\left(-\frac{g}{l}\sin(\xi_1)-\frac{k}{m}\xi_2+\frac{1}{ml^2}\upsilon\right)\diff{t}+0.03\xi_2\diff{W^1_t},
                \end{array}
                \right.
\end{equation}
where $\xi_1$ and $\xi_2$ represent the angular position and the velocity of the point mass on the pendulum, 
$\upsilon$ is the torque applied to the cart, 
$g=9.8$ is acceleration due to gravity, 
$l=0.5$ is the length of the shaft, 
$m=0.6$ is the mass, 
and $k=2$ is the friction coefficient. 
All the constants and the variables are considered in SI units. 
We assume that $\mathsf{U}=[-1.5,~1.5]$ and that $\mathcal{U}_\tau$ contains curves taking values in $[\mathsf{U}]_{0.5}$. We work on the subset $D=[-1,~1]\times[-1,~1]$ of the state space of $\Sigma$. 
Using the Lyapunov function $V(x,x')=(x-x')^TP(x-x')$, for all $x,x'\in\R^2$, proposed in \cite{pola4}, where $P=\left[
                \begin{array}{cc}
                  1.5&0.3\\ 
                 0.3&1.5
                \end{array}
                \right],$ 
it can be readily verified that $V$ satisfies conditions (i) and (ii) in Definition \ref{delta_PISS_Lya} with $\ul\alpha({r})=1.2\,r$, and $\ol\alpha({r})=3.6\,r$, for $q=2$. Moreover, $V$ satisfies condition (iii) in Definition \ref{delta_PISS_Lya} with $\kappa=0.7691$, and $\rho({r})=8.76\,r$, for all $x,x'\in D$, and again $q=2$. 
Therefore, $\Sigma$ is $\delta$-ISS-M$_2$, equipped with the $\delta$-ISS-M$_2$ Lyapunov function $V$, as long as we are interested in its dynamics, initialized on $D$. 
Using the results in Theorem \ref{the_Lya}, one obtains that functions $\beta(r,s)=3\mathsf{e}^{-\kappa s}$ and $\gamma({r})=3.49\,r$ satisfy property (\ref{delta_PISS}) for $\Sigma$. 
                
For a given fixed sampling time $\tau=3$, the precision $\varepsilon$ is lower bounded by $0.0778$ and $0.4848$ using the results in Theorems \ref{main_theorem2} and \ref{main_theorem}, respectively. Hence, the results in Theorem \ref{main_theorem2} provide a symbolic model which is much less conservative than the one provided by Theorem \ref{main_theorem}. By selecting an $\varepsilon=0.085$, the parameter $\eta$ for $S_{\params}(\Sigma)$ based on the results in Theorem \ref{main_theorem2} is obtained as $0.0033$. The resulting cardinality of state and input sets for $S_{\params}(\Sigma)$ are $370881$ and $7$, respectively. 
The CPU time taken for computing the abstraction $S_{\params}(\Sigma)$ has amounted to 17278 seconds.

Now, consider the objective to design a controller forcing the trajectories of $\Sigma$, starting from the initial condition $x_0=(0,~0)$, to first sequentially visit (in the $2{\textsf{nd}}$ moment metric) two regions of interest $W_1=[0.48,~0.58]\times[-1,~1]$, and $W_2=[-0.58,~-0.48]\times[-1,~1]$; then, once the system has visited the regions, to return to the first region $W_1$ and remain there forever. The LTL formula\footnote{Note that the semantics of LTL are defined over the output behaviors of $S_{\params}(\Sigma)$.} encoding this goal is $\Diamond\Box W_1\wedge \Diamond \left(W_1\wedge\Diamond W_2\right)$. The CPU time taken for synthesizing the controller has amounted to $3.02$ seconds. Figure \ref{fig1} displays a few realizations of the closed-loop solution process $\xi_{x_0\upsilon}$ stemming from the initial condition $x_0=(0,~0)$, as well as the corresponding evolution of the input signal. In Figure \ref{fig3}, we show the square root of the average value (over 100 experiments) of the squared distance in time of the solution process $\xi_{x_0\upsilon}$ to the sets $W_1$ and $W_2$, namely $\left\Vert\xi_{x_0\upsilon}(t)\right\Vert^2_{W_1}$ and $\left\Vert\xi_{x_0\upsilon}(t)\right\Vert^2_{W_2}$, where the point-to-set distance 
is defined as $\Vert{x}\Vert_{W}=\inf_{w\in{W}}\Vert x-w\Vert$.   
Notice that the square root of this empirical (averaged) squared distances is significantly lower than the selected bound on the precision $\varepsilon=0.085$, 
as expected since the conditions based on Lyapunov functions can lead to conservative bounds. 
(We have discussed that the bounds can be improved by seeking optimized Lyapunov functions.) 

Notice that using the chosen quantization parameters and \eqref{bisim_cond200}, we obtain a precision $\varepsilon=0.005$ in \eqref{simultaneous1} or \eqref{simultaneous2}.
Using the result in Theorem \ref{lemma6}, one can conclude that, as long as we are interested in dynamics of $\Sigma$ on $D$, we have an $\alpha=0.0032$ in Theorem \ref{lemma6}, and that the refinement of a controller satisfying an LTL formula $\varphi$ for $S_\params(\Sigma)$ (e.g. $\varphi=\Diamond\Box W_1\wedge \Diamond \left(W_1\wedge\Diamond W_2\right)$) to the system $S_\tau(\Sigma)$ satisfies the ``inflated formula'' $\varphi^{0.31}$ (e.g. $\varphi^{0.31}=\Diamond\Box W_1^{0.31}\wedge \Diamond \left(W_1^{0.31}\wedge\Diamond W_2^{0.31}\right)$) with probability at least $70\%$ over a discrete time horizon spanning $\{0,3,\cdots,27\}$ seconds,  where $\varphi^{0.31}$ is $0.31$-inflation of $\varphi$ as defined in \cite{liu}. 

Furthermore, employing the result in Theorem \ref{lemma2}, one can conclude that the refinement of a controller satisfying $\Diamond{W}$ for $S_\params(\Sigma)$ to  system $S_\tau(\Sigma)$ satisfies $\Diamond{W^\epsilon}$ with probability $1-\varepsilon/{\epsilon}$, where $W^\epsilon=\left\{x\in\R^n\,\,|\,\,\left\Vert{x}\right\Vert_W\leq\epsilon\right\}$. For example, refinement of a controller, satisfying $\Diamond{W_1}$ for $S_\params(\Sigma)$, to the system $S_\tau(\Sigma)$ satisfies $\Diamond{W_1^{0.28}}$ with probability at least $70\%$.

\begin{figure}[h]
\begin{center}
\includegraphics[width=12cm]{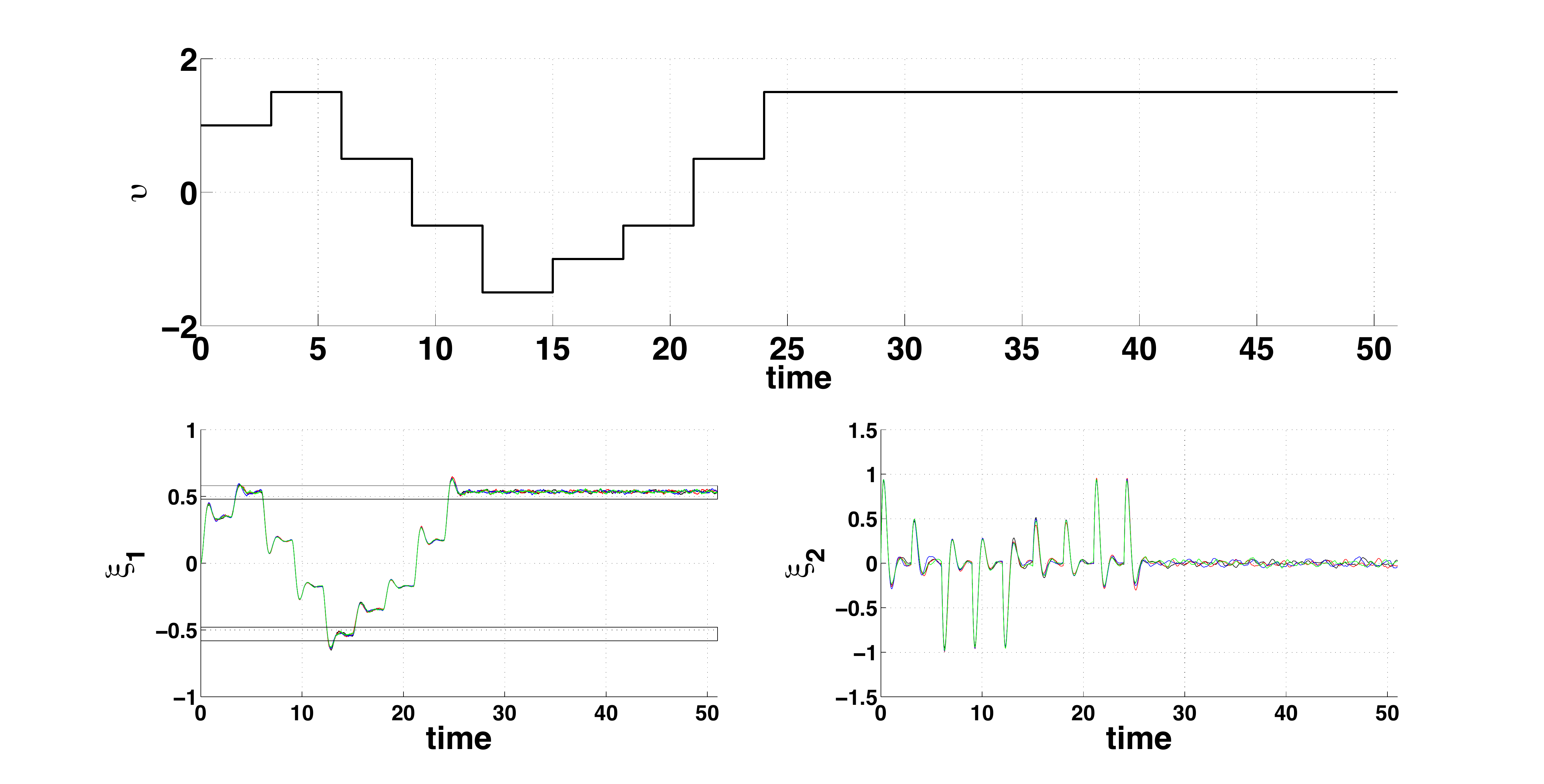}
\end{center}
\caption{Example 1: A few realizations of the closed-loop solution process $\xi_{x_0\upsilon}$ with initial condition \mbox{$x_0=(0,~0)$} (bottom panel) and the corresponding evolution of the obtained (unique) input signal $\upsilon$ (top panel).}
\label{fig1}
\end{figure}

\begin{figure}[h]
\begin{center}
\includegraphics[width=12cm]{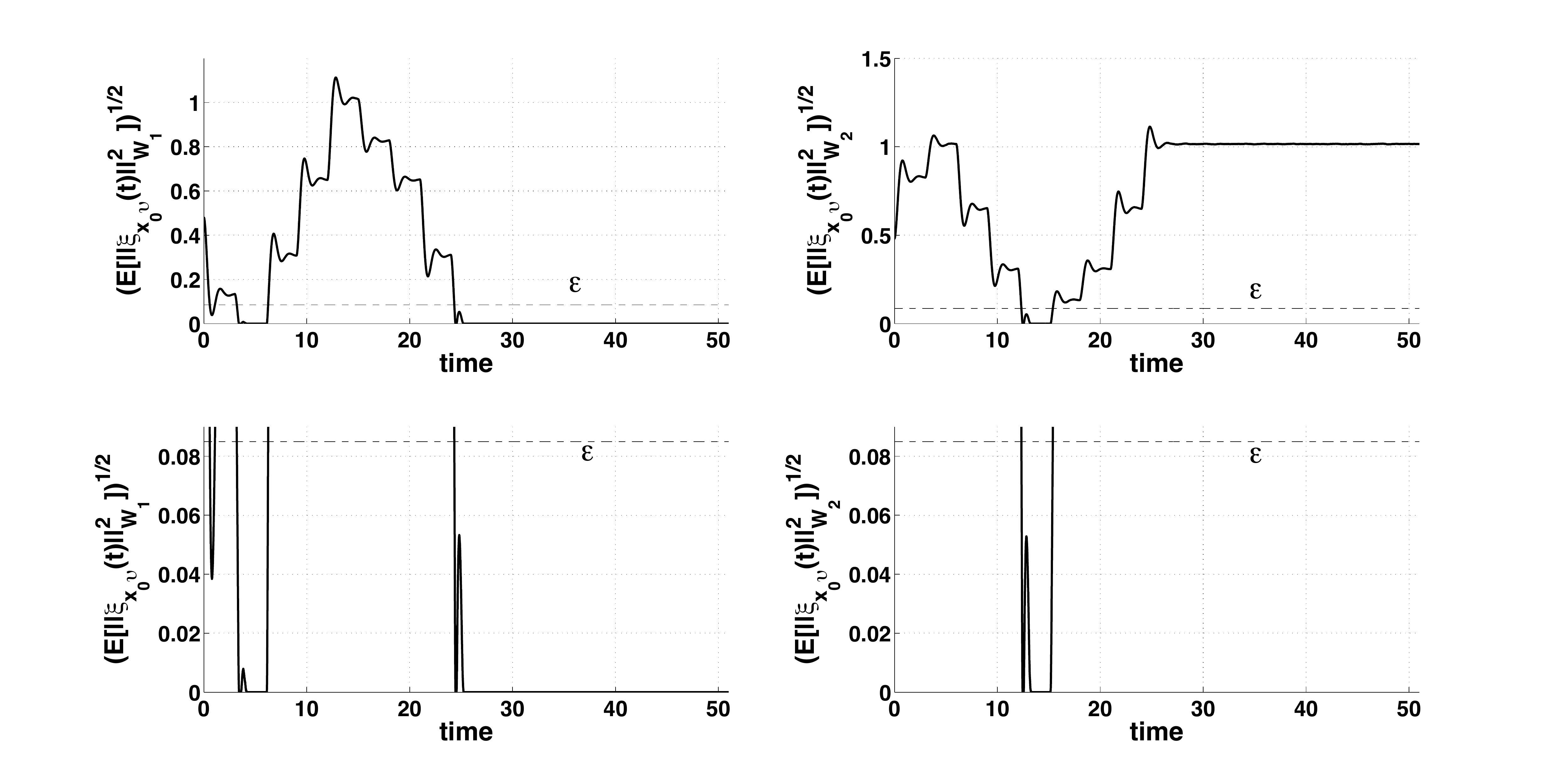}
\end{center}
\caption{Example 1: Square root of the average values (over 100 experiments) of the squared distance of the solution process $\xi_{x_0\upsilon}$ to the sets $W_1$ (left panels) and $W_2$ (right panels), in two different vertical scales (top vs bottom panels).}
\label{fig3}
\end{figure}

\subsection{Linear model, $1{\textsf{st}}$ moment, ``reach-and-stay, while staying'' property}
Consider a linear DC motor borrowed from \cite{CMU_Examples}, now affected by noise and described by:
\begin{equation}\label{DCmotor}
    \Sigma:\left\{
                \begin{array}{ll}
                  \diff{\xi}_1=\left(-\frac{b}{J}\xi_1+\frac{K}{J}\xi_2\right)\diff{t}+0.15\xi_1\diff{W^1_t},\\ 
                 \diff{\xi}_2=\left(-\frac{K}{L}\xi_1-\frac{R}{L}\xi_2+\frac{1}{L}\upsilon\right)\diff{t}+0.15\xi_2\diff{W^2_t}, 
                \end{array}
                \right.
\end{equation}
where $\xi_1$ is the angular velocity of the motor, 
$\xi_2$ is the current through an inductor, 
$\upsilon$ is the voltage signal, 
$b=10^{-4}$ is the damping ratio of the mechanical system, 
$J=25\times10^{-5}$ is the moment of inertia of the rotor, 
$K=5\times{10}^{-2}$ is the electromotive force constant, 
$L=3\times{10}^{-4}$ is the electric inductance, and 
$R=0.5$ is the electric resistance. 
All constants and variables refer to the SI units. 
It can be readily verified that the system $\Sigma$ satisfies condition (\ref{LMI}) with constant $\widehat\kappa=40$ and matrix 
$P=\left[
                \begin{array}{cc}
                  1.2201&0.2224\\ 
                 0.2224&1.2248
                \end{array}
                \right].$ 
Therefore, $\Sigma$ is $\delta$-ISS-M$_q$ and equipped with the $\delta$-ISS-M$_q$ Lyapunov function $V(x,x')$ in (\ref{V}), where $q\in\{1,2\}$. In this example, we use $q=1$.
                
We assume that $\mathsf{U}=[-0.5,~0.5]$ and that $\mathcal{U}_\tau$ contains curves taking values in $[\mathsf{U}]_{0.1}$. 
We work on the subset $D=[-5,~5]\times[-5,~5]$ of the state space of $\Sigma$.  
For a given precision $\varepsilon=1$ and fixed sampling time $\tau=0.01$, the parameter $\eta$ of $S_{\params}(\Sigma)$ based on the results in Theorem \ref{main_theorem} is equal to $0.01$. 
Note that for sampling times $\tau<0.026$, and in particular for a choice $\tau=0.01$, 
the results in Theorem \ref{main_theorem2} cannot be applied here because $\left(\beta\left(\varepsilon^q,\tau\right)\right)^{\frac{1}{q}}>\varepsilon$ and condition (\ref{bisim_cond2}) in Theorem \ref{main_theorem2} is not fulfilled. 
The resulting cardinality of the state and input sets for $S_{\params}(\Sigma)$ amounts to $1002001$ and $11$, respectively. 
The CPU time used for computing the abstraction has amounted to $148.092$ seconds. 

Now, consider the objective to design a controller forcing (in the $1{\textsf{st}}$ moment metric) the trajectories of $\Sigma$ to reach and stay within $W=[4.5,~5]\times[-0.25,~0.25]$, while ensuring that the current through the inductor is restricted between $-0.25$ and $0.25$. 
This corresponds to the LTL specification $\Diamond\Box W\wedge \Box Z$, where $Z=[-5,~5]\times[-0.25,~0.25]$. 
The CPU time used for synthesizing the controller has been of $3.88$ seconds. 
Figure \ref{fig11} displays a few realizations of the closed-loop solution process $\xi_{x_0\upsilon}$ stemming from the initial condition $x_0=(-5,~-0)$, as well as the corresponding evolution of the input signal. In Figure \ref{fig22}, we show the average value (over 100 experiments) of the distance in time of the solution process $\xi_{x_0\upsilon}$ to the sets $W$ and $Z$, namely $\left\Vert\xi_{x_0\upsilon}(t)\right\Vert_W$ and $\left\Vert\xi_{x_0\upsilon}(t)\right\Vert_{Z}$.   
Notice that the empirical average distances are as expected significantly lower than the precision $\varepsilon=1$.  

Note that using the selected quantization parameters in \eqref{bisim_cond10} and \eqref{bisim_cond0}, 
we obtain a value $\varepsilon=0.08$ in \eqref{simultaneous}, \eqref{simultaneous1} or \eqref{simultaneous2}.
Further, using the result in Theorem \ref{lemma5}, one obtains that, as long as we are interested in dynamics of $\Sigma$, initialized on $D$ and $\mathsf{U}=\{0_m\}$, there exists a state run $\mathsf{X}_{x_{\params0}u}$ of $S_\params(\Sigma)$ satisfying an LTL formula $\varphi$ if and only if there exists a state run $\xi_{x_{\tau0}\upsilon}$ of $S_\tau(\Sigma)$ satisfying $\varphi^{1}$ with probability at least $79\%$ over the infinite time horizon,  where $\varphi^{1}$ is the 
$1$-inflation of $\varphi$. 

Moreover, using the result in Theorem \ref{lemma6}, one can conclude that, as long as we are interested in dynamics of $\Sigma$ on $Z$, implying that $\alpha=0.6917$ in Theorem \ref{lemma6}, the refinement of a controller, satisfying an LTL formula $\varphi$ for $S_\params(\Sigma)$ (e.g. $\varphi=\Diamond\Box W\wedge \Box Z$) to the system $S_\tau(\Sigma)$ satisfies $\varphi^{1}$ (e.g. $\varphi^1=\Diamond\Box W^1\wedge \Box Z^1$) with probability at least $70\%$ over a discrete time horizon spanning $\{0,0.01,0.02,\cdots,1\}$ seconds.

\begin{figure}[h]
\begin{center}
\includegraphics[width=12cm]{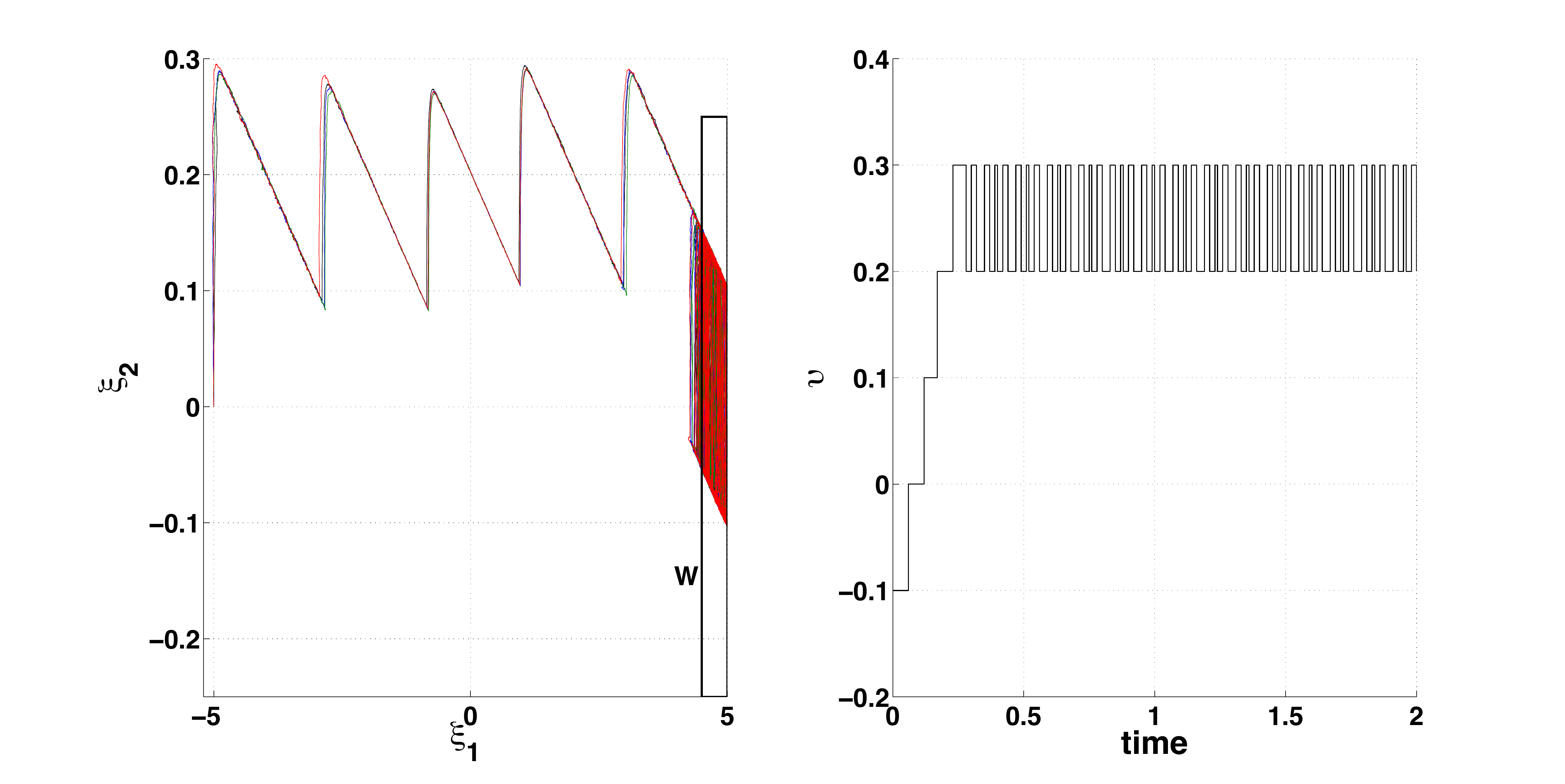}
\end{center}
\caption{Example 2: A few realizations of the closed-loop solution process $\xi_{x_0\upsilon}$ with initial condition \mbox{$x_0=(-5,~0)$} (left panel), 
and the evolution of the input signal $\upsilon$ (right panel).}
\label{fig11}
\end{figure}

\begin{figure}[h]
\begin{center}
\includegraphics[width=12cm]{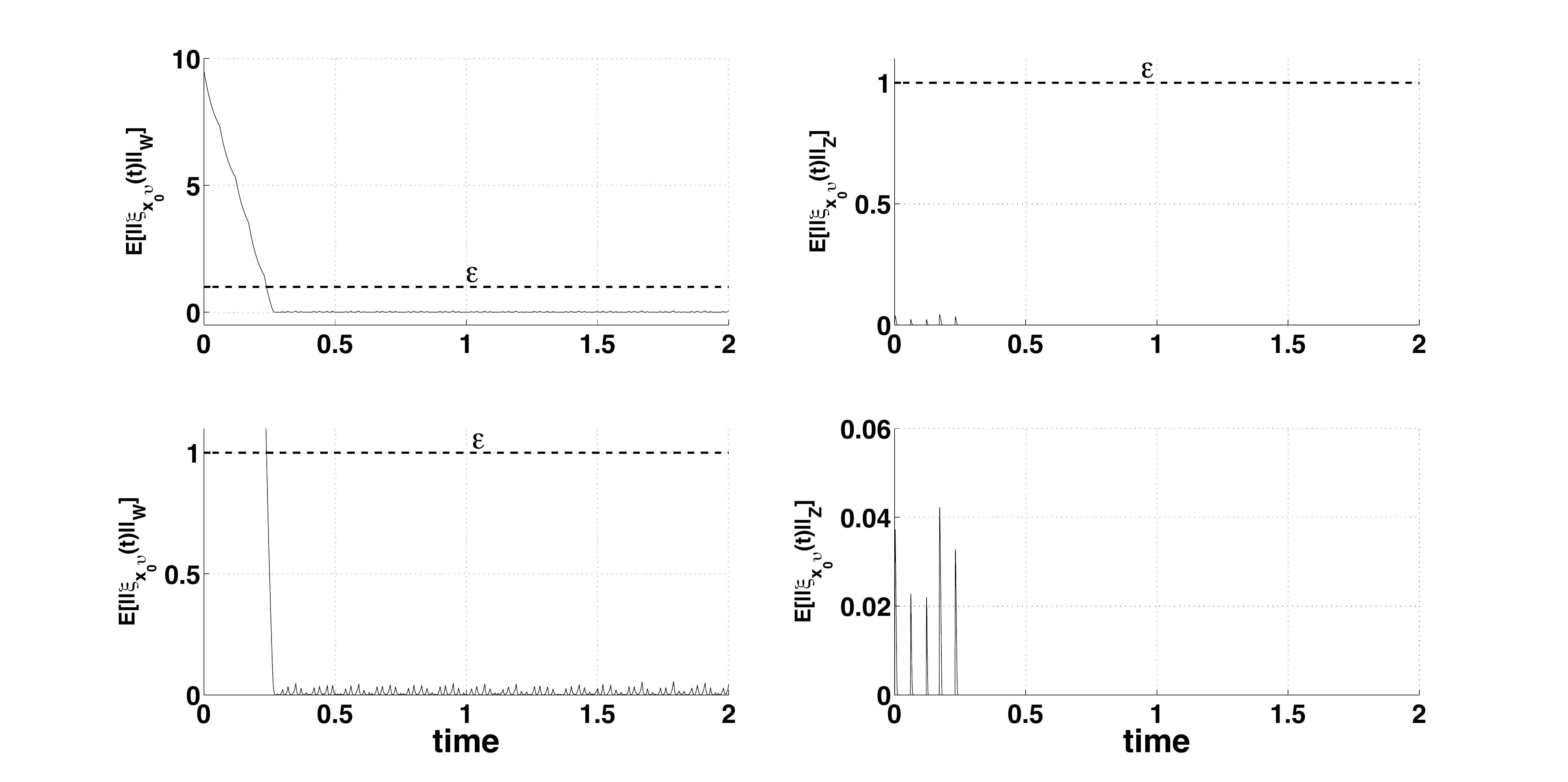}
\end{center}
\caption{Example 2: Average values (over 100 experiments) of the distance of the solution process $\xi_{x_0\upsilon}$ to the sets $W$ (left panels) and $Z$ (right panels), in two different vertical scales (top vs bottom panels).}
\label{fig22}
\end{figure}

\section{Conclusions}
This work has shown that any stochastic sampled-data control system, admitting a $\delta$-ISS-M$_q$ Lyapunov function of the form in (\ref{V}) or with a shape as in Lemma \ref{lemma3}, and initializing within a compact set of states, admits a finite approximately bisimilar symbolic model (in the sense of moments or probability). 
The constructed symbolic model can be used to synthesize controllers enforcing complex logic specifications, 
expressed via linear temporal logic or as automata on infinite strings. 

The main limitation of the design methodology developed in this paper lies in the cardinality of the set of states of the computed symbolic model. The authors are currently investigating several different techniques to address this limitation. Promising work over non-probabilistic control systems includes specification-guided abstractions \cite{matthias}, results on differentially flat systems \cite{alessandro}, and the use of non-uniform quantization \cite{tazaki}. Furthermore, the authors are currently working toward extensions of the results over general stochastic hybrid systems. 

\section{Acknowledgements}
The authors would like to thank Ilya Tkachev for fruitful technical discussions. 

\bibliographystyle{alpha}
\bibliography{reference}
\newpage
\section{Appendix}

\begin{proof}[Proof of Lemma \ref{lem:lyapunov}]
It is not difficult to check that the function $V$ in \eqref{V} satisfies properties (i) and (ii) of Definition \ref{delta_PISS_Lya} with functions $\ul{\alpha}(y) \Let\left(\frac{1}{q}\lambda_{\min}\left({P}\right)\right)^{\frac{q}{2}}y$ and $\ol{\alpha}(y) \Let \left(\frac{n}{q}\lambda_{\max}\left({P}\right)\right)^{\frac{q}{2}}y$. It then suffices to verify property (iii). We verify property (iii) for the case that $f$ is differentiable and using condition (\ref{nonlinear ineq cond1}). The proof, using condition (\ref{nonlinear ineq cond}), follows similarly by removing the inequalities in the proof including derivative of $f$. By the definition of $V$ in (\ref{V}), for any $x,x'\in\R^n$ such that $x\neq{x'}$, and for $q\in\{1,2\}$, one has	
	\begin{align*}
		\partial_xV&= -\partial_{x'} V=(x-x')^TP\left(\widetilde{V}(x,x')\right)^{\frac{q}{2}-1},\\
		\partial_{x,x}V&= \partial_{x',x'}V= -\partial_{x,x'}V 
		= P\left(\widetilde{V}(x,x')\right)^{\frac{q}{2}-1}+\frac{q-2}{q}P(x-x')(x-x')^TP\left(\widetilde{V}(x,x')\right)^{\frac{q}{2}-2}.
	\end{align*}	
	Therefore, following the definition of $\mathcal{L}^{u,u'}$, and for any \mbox{$x,x',z\in\R^n$} such that $x\neq{x'}$, and any $u,u'\in\mathsf{U}$, one obtains the chain of (in)equalities in \eqref{1}. In \eqref{1}, $z \in \R^n$ and the mean value theorem \cite{kolk} is applied to the differentiable function $x \mapsto f(x,u)$ at points $x,x'$ for a given input value $u \in \mathsf{U}$ and $L_u$ is the Lipschitz constant, as introduced in Definition \ref{Def_control_sys}. Therefore, the function $V$ in \eqref{V} satisfies property (iii) of Definition \ref{delta_PISS_Lya} with positive constant $\kappa=\frac{\widetilde\kappa}{q}$ and $\mathcal{K}_\infty$ function $\rho({r})=\frac{n^{\frac{q}{2}}L^q_u}{\widetilde\kappa^{q-1}}\left\Vert\sqrt{P}\right\Vert^q r^q$.
\begin{figure}[b]
--------------------------------------------------------------------------------------------------------------------------------------------
	\begin{small}
	\begin{align}\nonumber
		\mathcal{L}^{u,u'} V(x, x')&= (x-x')^TP\left(\widetilde{V}(x,x')\right)^{\frac{q}{2}-1}\left(f(x,u)-f(x',u')\right)+\frac{1}{2} \text{Tr} \left(\begin{bmatrix} \sigma(x) \\ \sigma(x') \end{bmatrix}\left[\sigma^T(x)~~\sigma^T(x')\right] \begin{bmatrix}
\partial_{x,x} V& -\partial_{x,x}V\\ -\partial_{x,x}V& \partial_{x,x}V
\end{bmatrix}\right)\\\notag
		& = (x-x')^TP\left(\widetilde{V}(x,x')\right)^{\frac{q}{2}-1}\left(f(x,u) - f(x',u')\right)+ \frac{1}{2}\text{Tr}\left( \left(\sigma(x)-\sigma(x')\right) \left ( \sigma^T(x) - \sigma^T(x') \right) \partial_{x,x} V \right)	\\\notag	
		& = (x-x')^TP\left(\widetilde{V}(x,x')\right)^{\frac{q}{2}-1}\left(f(x,u) - f(x',u')\right)+\frac{1}{2}\left \| \sqrt{P} \left( \sigma(x) - \sigma(x')\right) \right \|_F^2\left(\widetilde{V}(x,x')\right)^{\frac{q}{2}-1}\\\notag&\qquad +\frac{q-2}{q}\left \|(x-x')^TP\left( \sigma(x) - \sigma(x')\right) \right \|_F^2\left(\widetilde{V}(x,x')\right)^{\frac{q}{2}-2} \\\notag
		&\le(x-x')^TP\left(\widetilde{V}(x,x')\right)^{\frac{q}{2}-1} \left(f(x,u)-f(x',u)+f(x',u) - f(x',u')\right)\\\notag&\qquad+\frac{1}{2}\left \| \sqrt{P} \left( \sigma(x) - \sigma(x')\right) \right \|_F^2\left(\widetilde{V}(x,x')\right)^{\frac{q}{2}-1}\\\label{1}
		& \le (x-x')^TP\left(\widetilde{V}(x,x')\right)^{\frac{q}{2}-1} \partial_x f(z,u)(x-x')+  (x-x')^TP\left(\widetilde{V}(x,x')\right)^{\frac{q}{2}-1}\left(f(x',u) - f(x',u')\right)\\\notag
		& \qquad + \frac{1}{2}\left \| \sqrt{P} \left( \sigma(x) - \sigma(x')\right) \right \|_F^2\left(\widetilde{V}(x,x')\right)^{\frac{q}{2}-1} \nonumber \\\notag
		&\le \left((x-x')^TP \partial_x f(z,u)(x-x')+\frac{1}{2}\left \| \sqrt{P} \left( \sigma(x) - \sigma(x')\right) \right \|_F^2\right)\left(\widetilde{V}(x,x')\right)^{\frac{q}{2}-1}\\\notag
		&\qquad+\sqrt{n}\left((x-x')^TP(x-x')\right)^{\frac{1}{2}}\left\Vert\sqrt{P}\right\Vert L_u\left\Vert u-u'\right\Vert\left(\widetilde{V}(x,x')\right)^{\frac{q}{2}-1}\\\notag
		&\leq-\widetilde\kappa V(x,x')+(q-1)\frac{\widetilde\kappa}{2}V(x,x')+\frac{n^{\frac{q}{2}}L^q_u}{\widetilde\kappa^{q-1}}\left\Vert\sqrt{P}\right\Vert^q\left\Vert u-u'\right\Vert^q=-\frac{\widetilde\kappa}{q}V(x,x')+\frac{n^{\frac{q}{2}}L^q_u}{\widetilde\kappa^{q-1}}\left\Vert\sqrt{P}\right\Vert^q\left\Vert u-u'\right\Vert^q 
\end{align}
\end{small}
--------------------------------------------------------------------------------------------------------------------------------------------
\end{figure}
\end{proof}

\medskip

\begin{proof}[Proof of Lemma \ref{lemma3}]
	In the proof, we use the notation $\sigma\sigma^T(x)$ instead of $\sigma(x)\sigma^T(x)$ and $H(V)(x,x')$ for the Hessian matrix of $V$ at $(x,x')\in\R^{2n}$. We drop the arguments of $\partial_{x,x}{V}$, $\partial_{x}{V}$, $\partial_{x'}{V}$, and $H(V)$ for the sake of simplicity. 
	In view of Ito's formula, Jensen's inequality, and similar to calculations in Lemma \ref{lem:lyapunov}, we have 
	\begin{footnotesize}
	\begin{align}\nonumber
		\ul\alpha\left(\EE\left[\left\Vert\traj{\xi}{x}{\upsilon}(t)-\traj{\ol \xi}{x}{\upsilon}(t)\right\Vert^q\right]\right)&\le\EE\left[\ul\alpha\left(\left\Vert\traj{\xi}{x}{\upsilon}(t)-\traj{\ol \xi}{x}{\upsilon}(t)\right\Vert^q\right)\right]\le\EE\left[ V \left(\traj{\xi}{x}{\upsilon}(t), \traj{\ol \xi}{x}{\upsilon}(t) \right)\right] =\int_0^t \EE\left[\mathcal{L}^{\upsilon(s),\upsilon(s)}V \left(\traj{\xi}{x}{\upsilon}(s), \traj{\ol \xi}{x}{\upsilon}(s) \right)\right]ds\\\notag 
		&=\int_0^t\EE\left[\left[\partial_xV~~\partial_{x'}V\right] \begin{bmatrix} f\left(\traj{\xi}{x}{\upsilon}(s),\upsilon(s)\right)\\f\left(\traj{\ol \xi}{x}{\upsilon}(s),\upsilon(s)\right)\end{bmatrix}+\frac{1}{2}\Tr\left(\sigma\sigma^T\left(\traj{\xi}{x}{\upsilon}(s)\right)\partial_{x,x}V\right)\right]ds\\\notag
		&= \int_0^t \EE \Bigg [ \left[\partial_xV~~\partial_{x'}V\right] \begin{bmatrix} f\left(\traj{\xi}{x}{\upsilon}(s),\upsilon(s)\right)\\f\left(\traj{\ol \xi}{x}{\upsilon}(s),\upsilon(s)\right)\end{bmatrix} + \frac{1}{2} \text{Tr} \left(\begin{bmatrix} \sigma\left(\traj{\xi}{x}{\upsilon}(s)\right) \\ \sigma\left(\traj{\ol \xi}{x}{\upsilon}(s)\right) \end{bmatrix}\left[\sigma^T\left(\traj{\xi}{x}{\upsilon}(s)\right)~~\sigma^T\left(\traj{\ol \xi}{x}{\upsilon}(s)\right)\right]H(V)\right)\notag\\ 
		&\quad+ \frac{1}{2}\Tr\left( \sigma\sigma^T\left(\traj{\xi}{x}{\upsilon}(s)\right)\partial_{x,x}{V}\right) - \frac{1}{2} \text{Tr} \left(\begin{bmatrix} \sigma\left(\traj{\xi}{x}{\upsilon}(s)\right) \\ \sigma\left(\traj{\ol \xi}{x}{\upsilon}(s)\right) \end{bmatrix}\left[\sigma^T\left(\traj{\xi}{x}{\upsilon}(s)\right)~~\sigma^T\left(\traj{\ol \xi}{x}{\upsilon}(s)\right)\right]H(V)\right)\Bigg]ds \notag\\ 
		 \label{lem11} &\leq \int_0^t - \kappa \EE \left[ V \left(\traj{\xi}{x}{\upsilon}(s), \traj{\ol \xi}{x}{\upsilon}(s) \right)\right]ds \\\notag&\quad+ \int_0^t \EE \left[\frac{1}{2}\Tr\left( \sigma\sigma^T\left(\traj{\xi}{x}{\upsilon}(s)\right)\partial_{x,x}{V}\right) - \frac{1}{2} \text{Tr} \left(\begin{bmatrix} \sigma\left(\traj{\xi}{x}{\upsilon}(s)\right) \\ \sigma\left(\traj{\ol \xi}{x}{\upsilon}(s)\right) \end{bmatrix}\left[\sigma^T\left(\traj{\xi}{x}{\upsilon}(s)\right)~~\sigma^T\left(\traj{\ol \xi}{x}{\upsilon}(s)\right)\right]H(V)\right)\right]ds \\
		 \label{lem22} &\leq \int_0^t -\kappa \EE \left[ V (\traj{\xi}{x}{\upsilon}(s), \traj{\ol \xi}{x}{\upsilon}(s) )\right]ds+ \int_0^t \EE \left[ \frac{1}{2}\Tr\left( \sigma\sigma^T\left(\traj{\xi}{x}{\upsilon}(s)\right)\partial_{x,x}{V} \right)\right]ds \le \widehat{h}(\sigma,t)\mathsf{e}^{-\kappa t},  
	\end{align}
	\end{footnotesize}where the function $\widehat{h}$ can be computed as
		$\widehat{h}(t,\sigma) \Let {\int_0^t \EE \left[ \frac{1}{2}\left\Vert \sqrt{\partial_{x,x}{V} }\sigma \left(\traj{\xi}{x}{\upsilon}(s)\right)\right\Vert_F^2 \right]ds}$.
	Inequality \eqref{lem11} is a straightforward consequence of $V$ satisfying the condition (iii) in Definition \ref{delta_PISS_Lya}, and \eqref{lem22} follows from Gronwall's inequality. 
	Using the Lipschitz continuity assumption on the diffusion term $\sigma$, we get:
	\begin{align}\notag
	&\int_0^t \EE \left[ \frac{1}{2}\left\Vert \sqrt{\partial_{x,x}{V} }\sigma \left(\traj{\xi}{x}{\upsilon}(s)\right)\right\Vert_F^2 \right]ds\leq\frac{1}{2}\sup_{x,x'\in{D}}\left\{\left\Vert{\sqrt{\partial_{x,x}{V}(x,x')}}\right\Vert^2\right\}n\min\{n,p\}Z^2\int_0^t\EE\left[\left\Vert\xi_{x\upsilon}(s)\right\Vert^2\right]ds.
	\end{align}
	Since $V$ is a $\delta$-ISS-M$_q$ Lyapunov function, $q\geq2$, $f(0_n,0_m)=0_n$, $\sigma(0_n)=0_{n\times{p}}$, and using functions $\beta$ and $\gamma$ in (\ref{delta_PISS}), one can verify that:
	\begin{align}\nonumber
	\left(\EE\left[\left\Vert\xi_{x\upsilon}(t)\right\Vert^2\right]\right)^{\frac{q}{2}}&\leq\EE\left[\left(\left\Vert\xi_{x\upsilon}(t)\right\Vert^2\right)^{\frac{q}{2}}\right]=\EE\left[\left\Vert\xi_{x\upsilon}(t)\right\Vert^q\right]\leq\beta\left(\left\Vert{x}\right\Vert^q,t\right)+\gamma\left(\Vert\upsilon\Vert_\infty\right)\\\notag&\leq\beta\left(\sup_{x\in{D}}\left\{\left\Vert{x}\right\Vert^q\right\},t\right)+\gamma\left(\sup_{u\in{\mathsf{U}}}\left\{\Vert{u}\Vert\right\}\right),
	\end{align}
	and hence
	\begin{align}\nonumber
	\EE\left[\left\Vert\xi_{x\upsilon}(t)\right\Vert^2\right]\le\left(\beta\left(\sup_{x\in{D}}\left\{\left\Vert{x}\right\Vert^q\right\},t\right)+\gamma\left(\sup_{u\in{\mathsf{U}}}\left\{\Vert{u}\Vert\right\}\right)\right)^{\frac{2}{q}}.
	\end{align}
	Therefore, we have
	\begin{align}\nonumber
	\widehat{h}(t,\sigma)&\leq\frac{1}{2}\sup_{x,x'\in{D}}\left\{\left\Vert{\sqrt{\partial_{x,x}{V}(x,x')}}\right\Vert^2\right\}n\min\{n,p\}Z^2\int_0^t\left(\beta\left(\sup_{x\in{D}}\left\{\left\Vert{x}\right\Vert^q\right\},s\right)+\gamma\left(\sup_{u\in{\mathsf{U}}}\left\{\Vert{u}\Vert\right\}\right)\right)^{\frac{2}{q}}ds.
	\end{align}
	By defining: 
	\begin{small}
	\begin{align}\label{up_bound1}
	h(\sigma,t)=\ul\alpha^{-1}\Bigg(\frac{1}{2}\sup_{x,x'\in{D}}\left\{\left\Vert{\sqrt{\partial_{x,x}{V}(x,x')}}\right\Vert^2\right\}n\min\{n,p\}Z^2\mathsf{e}^{-\kappa t}\int_0^t\left(\beta\left(\sup_{x\in{D}}\left\{\left\Vert{x}\right\Vert^q\right\},s\right)+\gamma\left(\sup_{u\in{\mathsf{U}}}\left\{\Vert{u}\Vert\right\}\right)\right)^{\frac{2}{q}}ds\Bigg),
	\end{align}
	\end{small}
	we obtain
	$\EE\left[\left\Vert\traj{\xi}{x}{\upsilon}(t)-\traj{\ol \xi}{x}{\upsilon}(t)\right\Vert^q\right]\leq h(\sigma,t)$.
	It is not hard to observe that the proposed function $h$ meets the conditions of the lemma.
\end{proof}

\medskip

\begin{proof}[Proof of Lemma \ref{lem:moment est}]
	In the proof, we use the notation $\sigma\sigma^T(x)$ instead of $\sigma(x)\sigma^T(x)$ for the sake of simplicity. 
	In view of Ito's formula and similar to calculations in Lemma \ref{lemma3}, we have 
	\begin{align}\nonumber
		&\frac{1}{2}\lambda_{\min}(P)\EE\left[\left\Vert\traj{\xi}{x}{\upsilon}(t)-\traj{\ol \xi}{x}{\upsilon}(t)\right\Vert^2\right]\le  \frac{q}{2}\EE\left[ V^{\frac{2}{q}} \left(\traj{\xi}{x}{\upsilon}(t), \traj{\ol \xi}{x}{\upsilon}(t) \right)\right]  \\\notag 
		&= \int_0^t \EE \left[  \left( \traj{\xi}{x}{\upsilon}(s) - \traj{\ol \xi}{x}{\upsilon}(s) \right)^TP \left( f \left( \traj{\xi}{x}{\upsilon}(s),\upsilon(s)\right) - f \left( \traj{\ol\xi}{x}{\upsilon}(s),\upsilon(s)\right) \right)+ \frac{1}{2}\Tr \left( \sigma\sigma^T\left(\traj{\xi}{x}{\upsilon}(s)\right)P \right) \right] ds \\\notag
		&= \int_0^t \EE \bigg [ \left( \traj{\xi}{x}{\upsilon}(s) - \traj{\ol \xi}{x}{\upsilon}(s) \right)^TP \left( f \left( \traj{\xi}{x}{\upsilon}(s),\upsilon(s)\right) - f \left( \traj{\ol\xi}{x}{\upsilon}(s),\upsilon(s)\right) \right)\\\notag&\qquad+\frac{1}{2}\Tr\left(\sigma\sigma^T\left(\traj{\xi}{x}{\upsilon}(s) - \traj{\ol \xi}{x}{\upsilon}(s) \right)P \right)+ \frac{1}{2}\Tr\left( \sigma\sigma^T\left(\traj{\xi}{x}{\upsilon}(s)\right)P -  \sigma\sigma^T\left(\traj{\xi}{x}{\upsilon}(s) - \traj{\ol \xi}{x}{\upsilon}(s) \right)P \right)\bigg]ds \notag\\ 
		 \label{lem1} &\leq \int_0^t - \widetilde{\kappa} \EE \left[ V^{\frac{2}{q}} \left(\traj{\xi}{x}{\upsilon}(s), \traj{\ol \xi}{x}{\upsilon}(s) \right)\right]ds\\\notag&\qquad+ \frac{1}{2}\int_0^t \EE \left[ \Tr\left( \sigma\sigma^T\left(\traj{\xi}{x}{\upsilon}(s)\right)P -   \sigma\sigma^T\left(\traj{\xi}{x}{\upsilon}(s) - \traj{\ol \xi}{x}{\upsilon}(s) \right)P \right)\right]ds \\ \label{lem2}
		 &\leq \int_0^t - \widetilde{\kappa} \EE \left[ V^{\frac{2}{q}} \left(\traj{\xi}{x}{\upsilon}(s), \traj{\ol \xi}{x}{\upsilon}(s) \right)\right]ds+ \frac{1}{2}\int_0^t \EE \left[ \Tr\left( \sigma\sigma^T\left(\traj{\xi}{x}{\upsilon}(s)\right)P\right)\right]ds \le \widehat{h}(t,\sigma)\mathsf{e}^{\frac{-2\widetilde{\kappa} t}{q}},  
	\end{align}
	where the function $\widehat{h}$ can be computed as
		$\widehat{h}(t,\sigma) \Let \frac{1}{2}{\int_0^t \EE \left[ \left\Vert \sqrt{P}\sigma \big(\traj{\xi}{x}{\upsilon}(s)\big)\right\Vert_F^2 \right]ds}$.
	Inequality \eqref{lem1} is a straightforward consequence of \eqref{nonlinear ineq cond}, and \eqref{lem2} follows from Gronwall's inequality. 
	Using the Lipschitz continuity assumption on the diffusion term $\sigma$, we get:
	\begin{align}\notag
	&\frac{1}{2}\int_0^t \EE \left[ \left\Vert \sqrt{P}\sigma \big(\traj{\xi}{x}{\upsilon}(s)\big)\right\Vert_F^2 \right]ds\leq\frac{1}{2}\left\Vert{\sqrt{P}}\right\Vert^2n\min\{n,p\}Z^2\int_0^t\EE\left[\left\Vert\xi_{x\upsilon}(s)\right\Vert^2\right]ds.
	\end{align}
	Since $V^{\frac{2}{q}}$ is a $\delta$-ISS-M$_2$ Lyapunov function, $f(0_n,0_m)=0_n$, $\sigma(0_n)=0_{n\times{p}}$, and using functions $\beta$ and $\gamma$ in (\ref{delta_PISS}), obtained by the $\delta$-ISS-M$_q$ Lyapunov function $V^{\frac{2}{q}}$, one can verify that:
	\begin{align}\nonumber
	\EE\left[\left\Vert\xi_{x\upsilon}(t)\right\Vert^2\right]\leq\frac{n\lambda_{\max}({P})}{\lambda_{\min}({P})}\Vert{x}\Vert^2\textsf{e}^{\frac{-\widetilde{\kappa}{t}}{2}}+\frac{2n\left\Vert{\sqrt{P}}\right\Vert^2L_u^2}{\textsf{e}\widetilde{\kappa}^2\lambda_{\min}({P})}\Vert\upsilon\Vert_\infty^2.
	\end{align}
	Therefore, one obtains:
	\begin{align}\nonumber
	\widehat{h}(t,\sigma)&\leq\frac{\left\Vert\sqrt{P}\right\Vert^2n^2\min\{n,p\}Z^2}{\lambda_{\min}({P})\widetilde{\kappa}}\left(\lambda_{\max}({P})\left(1-\textsf{e}^{\frac{-\widetilde{\kappa} t}{2}}\right)\sup_{x\in{D}}\left\{\Vert{x}\Vert^2\right\}+\frac{\left\Vert\sqrt{P}\right\Vert^2L_u^2}{\textsf{e}\widetilde{\kappa}}\sup_{u\in\mathsf{U}}\left\{\Vert{u}\Vert^2\right\}t\right).
	\end{align}
	By defining: 
	\begin{align}\label{up_bound2}
	h(\sigma,t)=&\frac{2\left\Vert\sqrt{P}\right\Vert^2n^2\min\{n,p\}Z^2\mathsf{e}^{\frac{-2\widetilde{\kappa}t}{q}}}{\lambda_{\min}^2({P})\widetilde{\kappa}}\left(\lambda_{\max}({P})\left(1-\textsf{e}^{\frac{-\widetilde{\kappa} t}{2}}\right)\sup_{x\in D}\left\{\Vert{x}\Vert^2\right\}+\frac{\left\Vert\sqrt{P}\right\Vert^2L_u^2}{\textsf{e}\widetilde{\kappa}}\sup_{u\in\mathsf{U}}\left\{\Vert{u}\Vert^2\right\}t\right),
	\end{align}
	we obtain
	$\EE\left[\left\Vert\traj{\xi}{x}{\upsilon}(t)-\traj{\ol \xi}{x}{\upsilon}(t)\right\Vert^2\right]\leq h(\sigma,t)$.
	It is not hard to observe that the proposed function $h$ meets the conditions of the lemma.
\end{proof}

\medskip

\begin{proof}[Proof of Corollary \ref{corollary1}]
	Motivated by inequality \eqref{lem1}, one can obtain
	\begin{align}\nonumber
		\EE &\left[ \Tr\left(  \sigma\sigma^T\left(\traj{\xi}{x}{\upsilon}(s)\right)P  -   \sigma\sigma^T\left(\traj{\xi}{x}{\upsilon}(s) - \traj{\ol \xi}{x}{\upsilon}(s) \right)P \right) \right] \\\notag
		& = \EE \left[ \traj{\xi^T}{x}{\upsilon}(s)  \left( \sum\limits_{i = 1}^{p} \sigma^T_iP\sigma_i  \right)\traj{\xi}{x}{\upsilon}(s)- \left( \traj{\xi}{x}{\upsilon}(s)  - \traj{\ol \xi}{x}{\upsilon}(s) \right)^T \left( \sum\limits_{i = 1}^{p} \sigma^T_iP\sigma_i \right) \left( \traj{\xi}{x}{\upsilon}(s)  - \traj{\ol \xi}{x}{\upsilon}(s) \right)\right] \\\notag
		& \label{rem1} = \traj{\ol \xi^T}{x}{\upsilon}(s)  \left( \sum\limits_{i = 1}^{p} \sigma^T_iP\sigma_i  \right)\traj{\ol \xi}{x}{\upsilon}(s) \le n\lambda_{\max}\left(\sum\limits_{i = 1}^{p} \sigma^T_iP\sigma_i\right) \left\|\traj{\ol \xi}{x}{\upsilon}(s) \right\|^2,  
	\end{align}
	where $\traj{\ol \xi}{x}{\upsilon}$ satisfies the ODE $\dot{\ol{\xi}}_{x\upsilon}(t) = A\ol\xi_{x\upsilon}(t) + B\upsilon(t)$. 
	It can be readily verified that
	\begin{align}
		\left\Vert\traj{\ol \xi}{x}{\upsilon}(t)\right\Vert \le \left\Vert\mathsf{e}^{At}\right\Vert\left\Vert x\right\Vert+\left(\int_0^{t}\left\Vert\mathsf{e}^{As}B\right\Vert ds\right)\Vert\upsilon\Vert_\infty\le\left\Vert\mathsf{e}^{At}\right\Vert\sup_{x\in{D}}\left\{\left\Vert x\right\Vert\right\}+\left(\int_0^{t}\left\Vert\mathsf{e}^{As}B\right\Vert ds\right)\sup_{u\in\mathsf{U}}\left\{\Vert{u}\Vert\right\}.
	\end{align}
	The above approximation, 
	together with \eqref{lem1}, 
	leads to a more explicit bound in terms of the system parameters as follows:
	\begin{align}\label{up_bound3}
	\EE\left[\left\Vert\traj{\xi}{x}{\upsilon}(t)-\traj{\ol \xi}{x}{\upsilon}(t)\right\Vert^2\right] \leq\frac{n\lambda_{\max}\left(\sum\limits_{i = 1}^{p} \sigma^T_iP\sigma_i\right)\mathsf{e}^{-\widehat{\kappa} t}}{\lambda_{\min}({P})}\int_0^t\left(\left\Vert\mathsf{e}^{As}\right\Vert\sup_{x\in{D}}\left\{\left\Vert x\right\Vert\right\}+\left(\int_0^{s}\left\Vert\mathsf{e}^{Ar}B\right\Vert dr\right)\sup_{u\in\mathsf{U}}\left\{\Vert{u}\Vert\right\}\right)^2ds,
	\end{align}	
	where $\widehat\kappa=\frac{2\widetilde\kappa}{q}$.
\end{proof}

\end{document}